\documentclass[11pt]{article}

\usepackage{fullpage}
\usepackage{amsmath, amsthm, amsfonts, amssymb, amstext, mathrsfs, enumerate}
\usepackage{graphicx, ragged2e, lscape, framed, xcolor}
\usepackage{subfiles}

\usepackage{pgf, tikz, float, subcaption, pgfplots}
\usetikzlibrary{graphs, graphs.standard, decorations.pathreplacing, calligraphy}

\usepackage[symbols,nogroupskip,sort=none]{glossaries-extra}

\usepackage[T1]{fontenc}

\theoremstyle{plain}
\newtheorem{theorem}{Theorem}[section]
\newtheorem{lemma}[theorem]{Lemma}
\newtheorem{proposition}[theorem]{Proposition}
\newtheorem{conjecture}[theorem]{Conjecture}

\newtheorem{problem}[theorem]{Problem}
\newtheorem{claim}{Claim}[section]
\newtheorem*{remark}{Remark}

\numberwithin{equation}{section}
\allowdisplaybreaks

\newcommand{\affl}[3]{\noindent #1, Email: {\tt #2}\\ \textsc{#3}\\[1.5pt]}

\usepackage[pagebackref]{hyperref}
\hypersetup{
	colorlinks=true,
    urlcolor=purple,
	linkcolor=purple,
    citecolor=purple,
}

\DeclareMathOperator{\spec}{spec}

\title{Convex combination of first and second eigenvalues of trees}
\author{Hitesh Kumar \and Bojan Mohar \and Shivaramakrishna Pragada \and Hanmeng Zhan}
\date{}

\begin{document}
\maketitle
\begin{abstract}
For a graph $G$, let $\lambda_1(G)$ and $\lambda_2(G)$ denote the largest and the second largest adjacency eigenvalue of $G$. The sum $\lambda_1(G) + \lambda_2(G)$ is called the \emph{spectral sum} of $G$. We investigate the spectral sum of trees of order $n$ and determine the extremal trees that attain the maximum/minimum. Moreover, for any $\alpha \in [0,1]$, we describe the extremal trees which maximize the convex combination $\alpha \lambda_1 + (1-\alpha)\lambda_2$ in the class of $n$-vertex trees for sufficiently large $n$.
\end{abstract}



\section{Introduction}

We use standard graph theory notation and terminology. Throughout, assume that $n\ge 5$ unless stated otherwise. Let $G = (V(G), E(G))$ be a finite simple graph of order $n$. If two vertices $u, v\in V(G)$ are adjacent in $G$, we write $u\sim v$, else we write $u\nsim v$. The \textit{adjacency matrix} of $G$ is an $n\times n$ matrix $A(G) = [a_{uv}]$, where $a_{uv} = 1$ if $u\sim v$ and $a_{uv} = 0$, otherwise. The \emph{eigenvalues} of $G$ are the eigenvalues of $A(G)$. Since $A(G)$ is a real symmetric matrix, all eigenvalues of $A(G)$ are real and can be listed as  
\[\lambda_1(G) \geq \lambda_2(G) \geq \cdots \geq \lambda_n(G).\]
We denote the multiset $\{\lambda_1(G), \ldots, \lambda_n(G)\}$ by $\spec(G)$ and call it the \emph{spectrum} of $G$. Denote the class of trees of order $n$ by $\mathcal{T}(n)$. For $k_1,  k_2\geq 0$ and 
$\ell\geq 1,$ denote by $DC(k_1,k_2,\ell)$ the tree on $n=k_1+k_2+\ell$ vertices obtained by attaching $k_1$ and $k_2$ leaves to the two terminal vertices of a path on $\ell$ vertices, see Figure \ref{fig:double_comet}. We call such a tree a \emph{double comet}. We denote by $K_{1,n-1}$ and $P_n$ the star and the path on $n$ vertices. We denote the \emph{complement} of $G$ by $\overline{G}$.

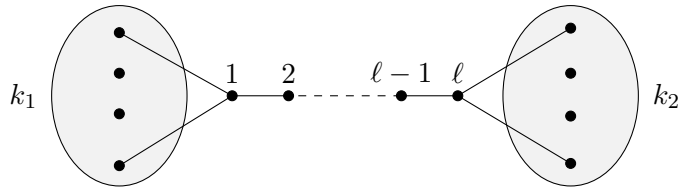
\begin{figure}
    \centering
		\begin{tikzpicture}[scale=1.5]
					\draw [rotate around={90:(2,0)}, fill=gray!10] (2,0) ellipse (0.8cm and 0.6cm);
					\draw [rotate around={90:(-2,0)}, fill=gray!10] (-2,0) ellipse (0.8cm and 0.6cm);
					\draw[dashed]  (0.5,0) -- (-0.5,0);
					\draw  (-1,0)-- (-2,-0.62);
					\draw  (-1,0)-- (-2,0.56);
					\draw  (0.5,0)-- (1,0);
                    \draw  (-0.5,0)-- (-1,0);
					\draw  (1,0)-- (2,0.6);
					\draw  (1,0)-- (2,-0.6);
					
					\draw [fill=black]
                    (0.5,0) circle (1.25pt)
				    (-0.5,0) circle (1.25pt)
				    (-1,0) circle (1.25pt)
				    (-2,-0.62) circle (1.25pt)
				    (-2,0.56) circle (1.25pt)
				    (1,0) circle (1.25pt)
				    (2,0.6) circle (1.25pt)
				    (2,-0.6) circle (1.25pt)
                    (-2,0.2) circle (1.25pt)
					(-2,-0.16) circle (1.25pt)
				    (2,0.2) circle (1.25pt)
				    (2,-0.18) circle (1.25pt);
                    \draw (2.85,0) node {$k_2$}
                    (-2.85,0) node {$k_1$}
                    (-1,0.2) node {$1$}
                    (-0.5,0.2) node {$2$}
                    (0.5,0.2) node {$\ell - 1$}
                    (1,0.2) node {$\ell$};
				\end{tikzpicture}
    \caption{$DC(k_1, k_2, \ell)$}
    \label{fig:double_comet}
\end{figure}

Extremizing graph eigenvalues, particularly $\lambda_1, \lambda_2$ and $\lambda_n$, over various graph families is an important problem in spectral graph theory, see the surveys \cite{Cvetkovic_Rowlinson_1990, Cvetkovic_Simic_1995, Stanic_book}. Investigation of combinations of eigenvalues such as Nordhaus-Gaddum type sums $\lambda_i(G) + \lambda_i(\overline{G})$ \cite{Nikiforov_2007, Aouchiche_Hansen_2013},  $\lambda_1 - \lambda_2$ (\emph{spectral gap} \cite{Stanic_book, Stanic_2013}), $\lambda_1+\lambda_2$ (which we call \emph{spectral sum}) \cite{Ebrahimi_Mohar_Nikiforov_Ahmady_2008}, $\lambda_1 - \lambda_n$ (\emph{spectral spread}) \cite{Breen_Riasanovsky_Tait_Urschel_2022}, and $\lambda_1 + \lambda_n$ (a measure of bipartiteness) \cite{Csikvari_2022}, has also led to interesting results, techniques and applications. 

A classic result of  Lov\'{a}sz and Pelik\'{a}n \cite{Lovasz_Pelikan_1973} concerns extremization of $\lambda_1$ for trees. 

\begin{theorem}[\cite{Lovasz_Pelikan_1973}]\label{thm:lambda_one_max_trees}
    If $T\in \mathcal{T}(n)$ and $T\notin \{K_{n-1, 1}, P_n\}$, then 
    \[ \lambda_1(K_{n-1, 1}) > \lambda_1(T) > \lambda_1(P_n). \]
\end{theorem}

The trees with minimum $\lambda_2$ are stars. The trees with maximum $\lambda_2$ were determined by Neumaier \cite{Neumaier_1982} and Hofmeister \cite{Hofmeister_1997}, see also \cite[Section 7]{KMPZ_trees_diameter_2025} 

\begin{theorem}[\cite{Neumaier_1982, Hofmeister_1997}] \label{thm:lambda_two_max_trees}
Let $T$ be a $\lambda_2$-maximizer in $\mathcal{T}(n)$. 
\begin{enumerate}[$(i)$]
    \item If $n$ is odd, then $T\in \left\{DC(\frac{n-3}{2}, \frac{n-3}{2}, 3), DC(\frac{n-3}{2}, \frac{n-5}{2}, 4), DC(\frac{n-5}{2}, \frac{n-5}{2}, 5)\right\}$.
    \item If $n$ is even, then $T\cong DC(\frac{n-4}{2}, \frac{n-4}{2}, 4)$.
\end{enumerate}
\end{theorem}

If one views stars and paths as extreme cases of double comets, one observes that the extremal trees that extremize $\lambda_1$ or $\lambda_2$ are always double comets. Naturally, one is tempted to conjecture that any meaningful combination of $\lambda_1$ and $\lambda_2$ will be extremized by a double comet in $\mathcal{T}(n)$. Indeed, it was conjectured by Jovovi\'{c}, Koledin and Stani\'{c} \cite{Stanic_2018} that the spectral gap is minimized by a double comet among trees. It is easy to see that $\lambda_1-\lambda_2$ is maximized by a star.

\begin{conjecture}[\cite{Stanic_2018}] \label{conj:spectral_gap_trees}
In $\mathcal{T}(n)$, the spectral gap $\lambda_1-\lambda_2$ is minimized by a double comet $DC(k,k,\ell)$ such that $2k+\ell=n.$
\end{conjecture}

In this paper, we focus on the spectral sum $\lambda_1+\lambda_2$, and more generally on the convex combination of these two eigenvalues for trees. Let $0\le \alpha \le 1$. For $T\in \mathcal{T}(n)$, define
\[ \Psi(T,\alpha)=\alpha\lambda_1(T)+(1-\alpha)\lambda_2(T).\]
Let 
\[ \Psi_n(\alpha)=\max\{\Psi(T,\alpha): T\in \mathcal{T}(n)\} \quad \text{ and }\quad \hat{\Psi}_n(\alpha)=\frac{\Psi_n(\alpha)}{\sqrt{n-1}}.\]
Then $\hat{\Psi}_n:[0,1]\rightarrow [0,1]$ for all $n$. Our main results are as follows. The cases $\alpha = 0$ and $\alpha =1$ are taken care of by Theorems \ref{thm:lambda_two_max_trees} and \ref{thm:lambda_one_max_trees}, respectively. In the range $ 0< \alpha < \frac{1}{2}$, we establish the following. 

\begin{theorem}\label{thm:alpha_less_than_half}
Fix $0 < \alpha < \frac{1}{2}$ and assume $n$ is sufficiently large (depending on $\alpha$). Let $T^*\in \mathcal{T}(n)$ be such that $\Psi(T^*,\alpha)=\Psi_n(\alpha)$. 
\begin{enumerate}[$(i)$]
    \item If $n$ is odd, then $T^* = DC(\frac{n-3}{2},\frac{n-3}{2}, 3)$.
    \item If $n$ is even, then
    \[T^* =  
    \begin{cases}
        DC(\frac{n-4}{2}, \frac{n-4}{2}, 4) & \text{ if } 0<\alpha < \frac{\sqrt{5}-1}{2\sqrt{5}} \approx 0.276; \\[4pt]
        DC(\frac{n-4}{2}, \frac{n-2}{2}, 3) & \text{ if } \frac{\sqrt{5}-1}{2\sqrt{5}} \le \alpha < \frac{1}{2}. 
    \end{cases}
    \]
\end{enumerate}
\end{theorem}

In the special case $\alpha = \frac{1}{2}$, we characterize the extremal tree for all $n$.

\begin{theorem}\label{thm:tree_max_spectral_sum} Suppose $n\ge 5$.
In the class $\mathcal{T}(n)$, the spectral sum $\lambda_1+\lambda_2$ is maximized by the unique tree $DC(\lfloor\frac{n-3}{2}\rfloor,\lceil\frac{n-3}{2}\rceil, 3)$. 
\end{theorem}

Next, we consider the range $\frac{1}{2}<\alpha < 1$. 

\begin{theorem}\label{thm:alpha_more_than_half}
Fix $\frac{1}{2}<\alpha <1$ and let $ t=\frac{\alpha^2}{\alpha^2 + (1-\alpha)^2}$. Assume $n$ is sufficiently large (depending on $\alpha$). Let $T^*\in \mathcal{T}(n)$ be such that $\Psi(T^*,\alpha)=\Psi_n(\alpha)$. Then 
  \[ T^*\in \{DC(k, n-2-k, 2): k = \lfloor\kappa_2\rfloor, \ \lceil \kappa_2\rceil \}\cup \{DC(k, n-3-k, 3): k = \lfloor\kappa_3\rfloor, \ \lceil \kappa_3\rceil \}\]
where 
\[ \kappa_2 := \frac{n-2+\sqrt{(2t-1)^2(n-1)^2-2(n-1)+1}}{2}\approx tn\]
and 
\[ \kappa_3 = \frac{n-3+\sqrt{(2t-1)^2(n-1)^2-4}}{2}\approx tn.\]
\end{theorem}

We interpret the above results as follows. In the range $0\le \alpha \le \frac{1}{2}$, $\lambda_2$ dictates the structure of the extremal trees, and hence they have two vertices with degree $\approx \frac{n}{2}$. In the range $\frac{1}{2}<\alpha\le 1$, the extremal tree has two vertices $u_1$ and $u_2$ with $\deg(u_1)\approx tn$ and $\deg(u_2)\approx (1-t)n$ such that $u_1$ and $u_2$ are at distance at most 2. One can view this as a continuous phenomenon. As the value of $\alpha$ increases, one obtains the extremal tree by moving leaves from one end (i.e., $u_2$) to the other end (i.e., $u_1$) of the double comet, and in the limit case (i.e., $\alpha \rightarrow 1$) one ends up getting a star. 

We also investigate the asymptotic behaviour of the normalized function $\hat{\Psi}_n(\alpha)$. This gives an upper/lower bound for the convex combination of $\lambda_1$ and $\lambda_2$ for trees. 

\begin{theorem}\label{thm:asymptotic_convex_combination}
    For any fixed $0\le \alpha \le 1$, we have 
    \[ \lim_{n\rightarrow \infty} \hat{\Psi}_n(\alpha) = 
    \begin{cases}
    \sqrt{\frac{1}{2}} & \text{ if }\alpha \le \frac{1}{2};\\
    \sqrt{\alpha^2 + (1-\alpha)^2} & \text{ if }\alpha > \frac{1}{2}.
    \end{cases}\]
\end{theorem}

It is only natural to ask which trees have the minimum spectral sum. 
We conclude the paper by answering this. The extremal tree is exactly what one would expect. 

\begin{theorem}\label{thm:spectral_sum_min_trees}
The spectral sum is uniquely minimized by the path $P_n$ among all trees in $\mathcal{T}(n)$ whenever $n\ge 16$.
\end{theorem}

\begin{remark}
    When $n\le 15$, the spectral sum is minimized by a star in $\mathcal{T}(n)$.
\end{remark} 

The paper is organized as follows. In Section \ref{section:preliminaries}, we recall known results that are required later. We prove Theorems \ref{thm:alpha_less_than_half} and \ref{thm:tree_max_spectral_sum} in Section \ref{section:convex_combination_alpha_less_than_half} and Theorem  \ref{thm:alpha_more_than_half} in Section \ref{section:convex_combination_alpha_more_than_half}. In Section \ref{section:asymptotics_convex_combination}, we prove Theorem \ref{thm:asymptotic_convex_combination}. The proof of Theorem \ref{thm:spectral_sum_min_trees} is given in Section \ref{section:minimum_spectral_sum}. We conclude with some remarks in Section \ref{section:conclusion}.

\section{Preliminaries}
\label{section:preliminaries}

In this section, we introduce additional notation and recall known results that will be required later. 

Let $G$ be a graph and $U\subseteq V(G)$. Then $G[U]$ denotes the subgraph of $G$ induced by $U$. For a vertex $u\in V(G)$, $N(u)$ and $N[u]$ will denote the open and closed neighbourhood of $u$ in $G$, respectively. We denote the maximum degree in $G$ by $\Delta(G)$. The distance between two vertices $u, v\in V(G)$ is denoted by $d(u, v)$. The \emph{characteristic polynomial} of $G$ is denoted by $\Phi(G,x) = \det(xI -A(G))$, where $I$ denotes the identity matrix of order $|V(G)|$.

In the following propositions, we find the characteristic polynomials, $\lambda_1$ and $\lambda_2$, of some special double comets. We require the following lemma.

\begin{lemma}[\cite{CRS_2010}]
\label{lemma:char_edge}
    Let $uv$ be an edge in a tree $T$. Then 
    \[\Phi(T,x)=\Phi(T-uv, x)-\Phi(T-u-v,x).\]
\end{lemma}

\begin{proposition}\label{prop:DC_path_2}
Let $T=DC(k_1, k_2,2)$ such that $k_1+k_2+2=n$. Then
\begin{enumerate}[$(i)$]
    \item $\lambda_1(T)=\sqrt{\frac{1}{2}\big(n-1+\sqrt{(n-1)^2 - 4k_1k_2} \big)}$.
    \item $\lambda_2(T)=\sqrt{\frac{1}{2}\big(n-1-\sqrt{(n-1)^2 - 4k_1k_2} \big)}$.
\end{enumerate}
\end{proposition}
\begin{proof}
  Using Lemma \ref{lemma:char_edge}, we get
  \[ \Phi(T,x)=x^{n-4}\big(x^4 - (n-1)x^2 + k_1k_2\big).\] Now, it is easy to see that the two largest roots of the above polynomial are as claimed.
\end{proof}

\begin{proposition}\label{prop:DC_path_3}
Let $T=DC(k_1,k_2,3)$ such that $k_1+k_2+3=n$. Then 
\begin{enumerate}[$(i)$]
    \item $\lambda_1(T)=\sqrt{\frac{1}{2}\big(n-1 + \sqrt{(k_1-k_2)^2+4}\big)}$.
    \item $\lambda_2(T)=\sqrt{\frac{1}{2}\big(n-1 - \sqrt{(k_1-k_2)^2+4}\big)}$.
\end{enumerate}
\end{proposition}
\begin{proof}
Using Lemma \ref{lemma:char_edge}, we have 
\[\Phi(T,x) = x^{n-4}\big(x^4-(n-1)x^2 + k_1k_2+k_1+k_2\big). \]
It is now easy to find the two largest roots of the above polynomial.
\end{proof}

We also recall the spectrum of the $n$-vertex path $P_n$. 

\begin{proposition}[\cite{Brouwer_Haemers_book}]\label{prop:path_spectrum}
    The spectrum of the path $P_n$ is given by $\{2\cos(\frac{\pi j}{n+1}): j = 1, \ldots, n\}$.
\end{proposition}

The following is a well-known result about symmetric matrices.

\begin{theorem}[Interlacing Theorem \cite{Horn_Johnson_2013}]
	Let $A$ be an $n\times n$ Hermitian matrix with eigenvalues $\lambda_1 \ge \cdots \ge \lambda_n$. Let $B$ be an $(n-1)\times (n-1)$ principal submatrix of $A$ with eigenvalues $\theta_1\ge \cdots \ge \theta_{n-1}$. Then $\lambda_i \ge \theta_i \ge \lambda_{i+1}$ for $1\le i \le n-1$.
\end{theorem}
 
Let $T$ be a tree. For a vector $y=(y_v)_{v\in V(T)} \in \mathbb{R}^{V(T)}$, define the \emph{support} of $y$ to be the set $S(y)=\{v \in V(T) : y_v\neq 0\}$. The \emph{positive support} of $y$ is $S^+(y) = \{v \in V(T) : y_v> 0\}$, the \emph{negative support} of $y$ is $S^-(y) = \{v \in V(T) : y_v< 0\}$, and the \emph{zero set} of $y$ is $S^0(y) = \{v \in V(T) : y_v= 0\}$. We recall the following result concerning the \emph{spectral center} of a tree.

\begin{theorem}[\cite{KMPZ_trees_diameter_2025}]\label{thm:spectral_center}
    Let $T$ be a tree of order $n\ge 2$. Let $y$ be a $\lambda_2(T)$-eigenvector with minimal support. Then there exist rooted subtrees $(H_1,a)$ and $(H_2,b)$ of $T$ such that $V(H_1) = S^+(y)$ and $V(H_2)=S^-(y)$. Furthermore, exactly one of the following holds. 
    \begin{enumerate}[$(i)$]
        \item If $S^0(y)\neq \emptyset$, then there is a unique vertex $v\in S^0(y)$ such that $(H_1,a)\circ v \circ (H_2,b)$ is a subtree of $T$. Moreover, 
        \[\lambda_1(H_1)=\lambda_1(H_2)=\lambda_1(T-v)=\lambda_2(T).\]  
        The vertex $v$ is called the \emph{spectral vertex}.
        \item If $S^0(y)=\emptyset$, then $T=(H_1,a)\circ (H_2,b)$, 
        \[\lambda_1(H_1-a) < \lambda_2(T) < \lambda_1(H_1) \text{ and } \lambda_1(H_2-b) < \lambda_2(T) < \lambda_1(H_2).\]
        The edge $ab$ is called the \emph{spectral edge}.
    \end{enumerate}
\end{theorem} 

The spectral vertex or edge is commonly referred to as the \emph{spectral center} of $T$. 

Next, we state some graph operations that increase (or decrease) the spectral radius.

\begin{lemma}\label{lemma:edge_addition} Let $G$ be a connected graph and $uv\notin E(G)$. Then $\lambda_1(G+uv)>\lambda_1(G)$. 
\end{lemma}

\begin{lemma}[Hanging Path Lemma \cite{Li_Feng_1979}, cf. Chapter 8 in \cite{CRS_2010}] \label{lemma:hanging_path} Let $G$ be a connected graph of order at least 2 and $v\in V(G)$. Denote by $G(k, \ell)$ the graph obtained from $G$ by attaching two paths of order $k$ and $\ell$ at $v$. If $k\ge \ell\ge 1$, then $\lambda_1(G(k, \ell)) > \lambda_1(G(k+1, \ell - 1))$.
\end{lemma}

Let $u$ and $v$ be distinct vertices in a graph $G$. Kelmans Operation on $G$ from $u$ to $v$ is defined as follows: replace each edge $ua$ with $va$ whenever $u\sim a\nsim v$ ($a$ distinct from $u,v$) in $G$ to obtain a new graph $G'$. Note that $G$ and $G'$ have the same number of vertices and edges. Moreover, we get the same graph $G'$ (up to isomorphism) when the roles of $u$ and $v$ are interchanged.   

\begin{lemma}[Kelmans Operation \cite{Csikvari_2009}]\label{lemma:Kelmans} Let $u$ and $v$ be vertices in a connected graph $G$ such that $N(u)\backslash \{v\} \nsubseteq N(v)\backslash \{u\}$ and $N(v)\backslash \{u\} \nsubseteq N(u)\backslash \{v\}$. If $G'$ is obtained from $G$ by applying Kelmans Operation from vertex $u$ to $v$, then $\lambda_1(G')>\lambda_1(G)$. 
\end{lemma}

We say a path $P=v_0 v_1\ldots v_s$ in a graph $G$ is an \emph{internal path} if $\deg(v_0)\ge 3$, $\deg(v_i)=2$ for $1\leq i \leq s-1$, and $\deg(v_s)\ge 3$. Let $uv$ be an edge in $G$. The graph $G'$ obtained from $G$ by \emph{contracting} the edge $uv$ is defined as follows: we replace each edge $ua$ with $va$ whenever $u\sim a \nsim v$ ($a$ distinct from $u,v$), and then we delete the vertex $u$. 

\begin{lemma}[Contraction Lemma \cite{Mckee_2018}]\label{lemma:Contraction} Let $uv$ be an edge on an internal path in a connected graph $G$ such that $u$ and $v$ have no common neighbours. Obtain $G'$ from $G$ by contracting the edge $uv$. Then $\lambda_1(G')\ge \lambda_1(G)$ and equality holds if and only if $\lambda_1(G)=2$.
\end{lemma}

The following is a consequence of the Courant-Fischer Theorem, as observed in \cite{Ebrahimi_Mohar_Nikiforov_Ahmady_2008}. 

\begin{lemma}[\cite{Ebrahimi_Mohar_Nikiforov_Ahmady_2008}]\label{lemma:spectral_sum}
For any graph $G$, we have
    \[\lambda_1(G) + \lambda_2(G) = \max_{\substack{||x||=||y||=1,\\ x^Ty = 0}} x^TA(G)x + y^TA(G)y.\]
\end{lemma}

Next, we describe a rotation operation in a tree which increases the value of $\Psi(T, \alpha)$. 

\begin{lemma}\label{lemma:rotation_alpha}
   Let $T$ be a tree and $x,y$ denote its unit $\lambda_1$-eigenvector and $\lambda_2$-eigenvector. Suppose $u, v, w\in V(T)$ are distinct vertices such that $u\sim v$ and $v\sim w$. Define $T' = T - vw + uw$. Fix $\frac{1}{2}\le \alpha \le 1$. If \[\alpha x_w(x_u - x_v) + (1-\alpha)y_w(y_u - y_v)>0,\] then 
   $\Psi(T', \alpha) > \Psi(T, \alpha)$.
\end{lemma}

\begin{proof}
    Using Lemma \ref{lemma:spectral_sum}, we have 
    \begin{align*}
        \Psi(T', \alpha) & = \alpha \lambda_1(T') + (1-\alpha)\lambda_2(T')\\
         &  =\left(\alpha - (1-\alpha)\right)\lambda_1(T')  + (1-\alpha)\left(\lambda_1(T') + \lambda_2(T')\right)\\
         & \ge \left(\alpha - (1-\alpha)\right)\left(\lambda_1(T) + 2x_w(x_u - x_v)\right) \\
         & \quad + (1-\alpha)\left(\lambda_1(T) + \lambda_2(T) + 2x_w(x_u - x_v) + 2y_w(y_u - y_v)\right)\\
         & = \Psi(T, \alpha) + 2\alpha x_w(x_u - x_v) + 2(1-\alpha)y_w(y_u - y_v)\\
         & > \Psi(T, \alpha). \qedhere
    \end{align*}
\end{proof}

We also note here some consequences of the eigenvalue-eigenvector equation, which will often be used later.

\begin{lemma}\label{lemma:eigenvalue_equations}
Let $T$ be a tree with $\lambda$-eigenvector $z$. For any $u\in V(T)$, the following equations hold.
   \begin{enumerate}[$(i)$]
       \item $\lambda^2z_u = \deg(u)z_u + \sum_{d(u,w) = 2} z_w$. 
       \item $ \lambda\left(\lambda^2 - \deg(u)\right)z_u = \sum_{v\in N(u)}z_v (\deg(v)-1) + \sum_{d(u,v) = 3}z_v$.
   \end{enumerate}
\end{lemma}

\begin{proof}
Using the eigenvalue-eigenvector equation, we have
\begin{align*}
   \lambda^2z_u & = \lambda \left(\sum_{v\in N(u)}z_v\right)  = \sum_{v\in N(u)} \sum_{w\in N(v)} z_w\\
   & = \deg(u)z_u + \sum_{v\in N(u)} \sum_{\substack{w\in N(v)\\ w\neq u}} z_w = \deg(u)z_u + \sum_{d(w,u)=2} z_w.   
\end{align*}
This proves $(i)$. Now, using $(i)$, we get
\begin{align*}
 \lambda\left(\lambda^2 - \deg(u)\right)z_u & = \lambda\sum_{d(u,w) = 2} z_w\\
 & = \sum_{d(u,w) = 2} \sum_{v\in N(w)} z_v\\
 & = \sum_{v\in N(u)}z_v (\deg(v)-1) + \sum_{d(u,v) = 3}z_v. \qedhere
\end{align*}
\end{proof}

Finally, we require a result determining the extremal tree in $\mathcal{T}(n)$ with second maximum $\lambda_1$. 

\begin{proposition}[\cite{Hofmeister_1997}, cf. Section 2 in \cite{KMPZ_trees_diameter_2025}]\label{prop:second_largest_lambda_one_trees} Let $n\ge 3$ and $T\in \mathcal{T}(n)$. If $T\neq K_{1,n-1}$, then 
\[ \lambda_1(T)\le \sqrt{\frac{1}{2}\big(n-1+\sqrt{n^2-6n+13}\big)}\] and equality holds if and only if $T\cong DC(n-3, 1, 2)$. 
\end{proposition}  

\section{Maximizing $\Psi_n(\alpha)$ when $0\le \alpha \le \frac{1}{2}$}
\label{section:convex_combination_alpha_less_than_half}

We first determine the extremal tree(s) in $\mathcal{T}(n)$ with second maximum $\lambda_2$ when $n$ is even.

\begin{proposition}\label{prop:second_maximum_lambda_2_even} Let $n\ge 12$ be even and $T\in \mathcal{T}(n)$. If $T\neq DC(\frac{n-4}{2}, \frac{n-4}{2}, 4)$, then 
\[ \lambda_2(T)\le \sqrt{\frac{n-1-\sqrt{5}}{2}},\]
and equality holds if and only if $T\cong DC(\frac{n-4}{2}, \frac{n-2}{2}, 3)$.
\end{proposition}
\begin{proof} Let $T$ be a $\lambda_2$-maximizer in $\mathcal{T}(n)\backslash \{DC(\frac{n-4}{2}, \frac{n-4}{2}, 4)\}$. Let $H=DC(\frac{n-4}{2}, \frac{n-2}{2}, 3)$. By Proposition \ref{prop:DC_path_3}, we have $\lambda_2(H)=\sqrt{\frac{n-1-\sqrt{5}}{2}}$. Suppose $T\neq H$. We will argue that $\lambda_2(T)<\lambda_2(H)$.

Let $H_1$ and $H_2$ be subtrees of $T$ as given in Theorem \ref{thm:spectral_center}. If $|H_i|\le \frac{n-2}{2}$ for some $i$, then  
\[ \lambda_2(T)\le \lambda_1(H_i)\le \lambda_1(K_{1, |H_i|-1})\le \sqrt{\frac{n-4}{2}} < \lambda_2(H).\]
So we can assume that $|H_i|=\frac{n}{2}$ for $i=1,2,$ and $T$ has a spectral edge. Now, if at least one $H_i$ is not a star, then by Proposition \ref{prop:second_largest_lambda_one_trees},
\[\lambda_2(T)<\lambda_1(H_i)\le \sqrt{\frac{1}{2}\bigg(\frac{n}{2}-1+\sqrt{\frac{n^2}{4}-3n+13}\bigg)} \le \sqrt{\frac{n-1-\sqrt{5}}{2}}\] 
for $n\ge 12$. So let $H_i$ be a star for $i=1,2$. Then 
\[T\in \left\{DC\left(\frac{n-2}{2}, \frac{n-2}{2}, 2\right), DC\left(\frac{n-4}{2}, \frac{n-2}{2}, 3\right), DC\left(\frac{n-4}{2}, \frac{n-4}{2}, 4\right)\right\}.\] 
By the choice of $T$, we have $T=DC(\frac{n-2}{2}, \frac{n-2}{2}, 2)$. Using Proposition \ref{prop:DC_path_2}, 
\[\lambda_2(T) = \sqrt{\frac{n-1-\sqrt{2n-3}}{2}} < \sqrt{\frac{n-1-\sqrt{5}}{2}}.\]
The proof is complete.
\end{proof}

We are now ready to prove Theorem \ref{thm:alpha_less_than_half}. 

\begin{proof}[Proof of Theorem \ref{thm:alpha_less_than_half}]
Fix $0 < \alpha \le \frac{1}{2}$ and assume $n\ge 12$. Let $T^*\in \mathcal{T}(n)$ be such that $\Psi(T^*, \alpha) = \Psi_n(\alpha)$. We denote $\lambda_1(T^*)$ and $\lambda_2(T^*)$ by $\lambda_1^*$ and $\lambda_2^*$, respectively.

We know that $(\lambda_1^*)^2 + (\lambda_2^*)^2\le n-1$. Define 
\[f(\lambda, \alpha) = \alpha\sqrt{n-1-\lambda^2} + (1-\alpha)\lambda\]
(imagine $\lambda$ is $\lambda_2^*$). The derivative of $f$ w.r.t. $\lambda$ is given by 
    \[\frac{\partial f}{\partial \lambda} = \frac{-\alpha \lambda}{\sqrt{n-1-\lambda^2}} + (1-\alpha).\]
We see that $\frac{\partial f}{\partial \lambda}\ge 0$ if and only if 
    \begin{equation}\label{eq:f(x,a)_bound}
       \lambda\le \sqrt{\frac{(n-1)(1-\alpha)^2}{\alpha^2 + (1-\alpha)^2}}. 
    \end{equation}
The right-hand side in \eqref{eq:f(x,a)_bound} is a decreasing function in $\alpha$, and equals $\sqrt{\frac{n-1}{2}}$ when $\alpha = \frac{1}{2}$. We conclude that for any fixed $0<\alpha \le \frac{1}{2}$, $f(\lambda, \alpha)$ is an increasing function in $\lambda$ in the interval $[0,\sqrt{\frac{n-1}{2}}]$. 

Let $H = DC(\lfloor\frac{n-3}{2}\rfloor,\lceil\frac{n-3}{2}\rceil, 3)$.
Now, by the choice of $T^*$, we have
\begin{equation}\label{eq:function_f}
    f(\lambda_2^*, \alpha) \ge \alpha \lambda_1^* + (1-\alpha) \lambda_2^* \ge  \alpha \lambda_1(H) + (1-\alpha)\lambda_2(H) = f(\lambda_2(H), \alpha),
\end{equation}
where the last equality holds since $H$ has only two positive eigenvalues by Proposition \ref{prop:DC_path_3}. We consider the following cases:

\textbf{Case 1:} $n$ is odd. 

Using Theorem \ref{thm:lambda_two_max_trees}, $H$ is a $\lambda_2$-maximizer and $\lambda_2(H) = \sqrt{\frac{n-3}{2}}$ by Proposition \ref{prop:DC_path_3}. Using \eqref{eq:function_f} and the fact that $f$ is an increasing function in $\lambda$ on the interval $[0,\sqrt{\frac{n-1}{2}}]$, we see that $\lambda_2^*\ge \lambda_2(H)$. Thus, $T^*$ is a $\lambda_2$-maximizer. By Theorem \ref{thm:lambda_two_max_trees}, we conclude that 
\[T^*\in \left\{DC\left(\frac{n-3}{2}, \frac{n-3}{2}, 3\right),\ DC\left(\frac{n-3}{2}, \frac{n-5}{2}, 4\right),\ DC\left(\frac{n-5}{2}, \frac{n-5}{2}, 5\right)\right\}.\] Using Kelmans Operation (Lemma \ref{lemma:Kelmans}), we see that $T^*\cong H$.

\textbf{Case 2:} $n$ is even. 

Then $H =  DC(\frac{n-4}{2}, \frac{n-2}{2}, 3)$. By Proposition \ref{prop:DC_path_3}, we have $\lambda_2(H) = \sqrt{\frac{n-1-\sqrt{5}}{2}}$. As in Case 1, we can argue that $\lambda_2^*\ge \lambda_2(H)$. By Proposition \ref{prop:second_maximum_lambda_2_even}, we conclude that 
\[T^*\in \left\{H, \ T\right\},\] 
where $T = DC(\frac{n-4}{2}, \frac{n-4}{2}, 4)$. 

Define 
\[ \rho := \sqrt{\frac{n-2}{2}}.\]
Using Proposition \ref{prop:DC_path_3} and Taylor expansion, we have 
\[\lambda_1(H)=\sqrt{\frac{n-1+\sqrt{5}}{2}} = \rho + \frac{1+\sqrt{5}}{4\rho} + O\left(\frac{1}{\rho^3}\right)\]
and 
\[\lambda_2(H) = \sqrt{\frac{n-1-\sqrt{5}}{2}} = \rho +  \frac{1 -\sqrt{5}}{4\rho} + O\left(\frac{1}{\rho^3}\right).\]
Thus, 
\begin{equation}\label{eq:H_Psi}
   \Psi(H, \alpha) = \rho + \frac{1+(2\alpha-1) \sqrt{5}}{4\rho} + O\left(\frac{1}{\rho^3}\right). 
\end{equation}

Again, using Lemma \ref{lemma:char_edge}, we get 
\[\Phi(T, x) = x^{n-4}\left( x^2 - \frac{n-2}{2}\right)^2 - x^{n-6}\left( x^2 - \frac{n-4}{2}\right)^2.\]
Thus, using standard root approximation techniques, we get
\[ \lambda_1(T) = \rho + \frac{1}{2\rho^2} + O\left(\frac{1}{\rho^3}\right)\quad \text{and}\quad \lambda_2(T) = \rho - \frac{1}{2\rho^2} + O\left(\frac{1}{\rho^3}\right).\]
Thus, 
\begin{equation}\label{eq:T_Psi}
  \Psi(T,\alpha) = \rho + \frac{(2\alpha-1)}{2\rho^2} + O\left(\frac{1}{\rho^3}\right).  
\end{equation}

Using \eqref{eq:H_Psi} and \eqref{eq:T_Psi}, we see that 
\[ \Psi(H,\alpha) - \Psi(T, \alpha) = \frac{1+(2\alpha-1) \sqrt{5}}{4\rho} - \frac{(2\alpha-1)}{2\rho^2}  +  O\left(\frac{1}{\rho^3}\right). \]
If $\alpha > \frac{\sqrt{5}-1}{2\sqrt{5}}$, then clearly $\Psi(H,\alpha) - \Psi(T, \alpha) > 0$ for sufficiently large $n$. Moreover, when $\alpha = \frac{\sqrt{5}-1}{2\sqrt{5}}$, we have 
\[ \Psi(H,\alpha) - \Psi(T, \alpha) = \frac{1}{2\sqrt{5}\ \rho^2}  +  O\left(\frac{1}{\rho^3}\right)>0\]
for large enough $n$. We conclude that 
\[ T^* = 
\begin{cases}
    H & \text{ if }\alpha \ge \frac{\sqrt{5}-1}{2\sqrt{5}};\\
    T & \text{ if }\alpha < \frac{\sqrt{5}-1}{2\sqrt{5}}.
\end{cases}\]
The proof is complete.
\end{proof}

Now, in the case $\alpha = \frac{1}{2}$, we can determine the extremal tree for all $n$. 

\begin{proof}[Proof of Theorem \ref{thm:tree_max_spectral_sum}] Assume $n\ge 12$; for small $n$ the assertion can be verified using SageMath (see \href{https://github.com/Shivaramkratos/Codes_misc/blob/main/convex_sum_specific_check.sage}{here} for the code).

From the proof of Theorem \ref{thm:alpha_less_than_half}, we already know that 
\[ T^* = 
\begin{cases}
  DC(\frac{n-3}{2},\frac{n-3}{2}, 3) & \text{ if }n \text{ is odd};\\
  DC(\frac{n-4}{2}, \frac{n-4}{2}, 4) \text{ or } DC(\frac{n-4}{2}, \frac{n-2}{2}, 3) & \text{ if } n \text{ is even}. 
\end{cases} \]
So we only need to consider the $n$ even case. Let $H = DC(\frac{n-4}{2}, \frac{n-2}{2}, 3)$ and $T =  DC(\frac{n-4}{2}, \frac{n-4}{2}, 4)$. 

From the proof of Theorem \ref{thm:alpha_less_than_half}, we know that 
\[ \lambda_1(H) + \lambda_2(H) = \sqrt{\frac{n-1+\sqrt{5}}{2}} + \sqrt{\frac{n-1-\sqrt{5}}{2}}. \]
Using Lemma \ref{lemma:char_edge}, we obtain 
\[\Phi(T, x) = x^{n-4}\left( x^2 - \frac{n-2}{2}\right)^2 - x^{n-6}\left( x^2 - \frac{n-4}{2}\right)^2.\]

It is easy to check that $\Phi(T, x)> 0$ for all $x\ge \sqrt{\frac{n-1}{2}}$ whenever $n\ge 20$. It follows that 
\[\lambda_1(T) < \sqrt{\frac{n-1}{2}}.\] Using the Interlacing Theorem, we see that
 \[ \lambda_2(T)\le \sqrt{\frac{n-2}{2}}.\]
Thus, for $n\ge 20$, we have
\[ \lambda_1(T) + \lambda_2(T) \le \sqrt{\frac{n-1}{2}} + \sqrt{\frac{n-2}{2}} < \sqrt{\frac{n-1+\sqrt{5}}{2}} + \sqrt{\frac{n-1-\sqrt{5}}{2}} = \lambda_1(H) + \lambda_2(H).\]

If $12\le n<20$, the assertion is easily verified using SageMath as we only need to compare the spectral sum of $H$ and $T$. This completes the proof. 
\end{proof}

\section{Maximizing $\Psi_n(\alpha)$ when $\frac{1}{2} < \alpha < 1$}
\label{section:convex_combination_alpha_more_than_half}

\subsection{Special double comets}

We first define some special double comets and determine their convex spectral sums. Fix $\alpha \in (\frac{1}{2}, 1)$ and assume $n$ is sufficiently large (depending on $\alpha$). Throughout this proof, let  
\[t:=\frac{\alpha^2}{\alpha^2 + (1-\alpha)^2}.\] 

\begin{proposition}\label{prop:DC_path_2_alpha}
Fix $\alpha \in (\frac{1}{2}, 1)$. Suppose that 
\[ \Psi(DC(\ell,n-2-\ell, 2), \alpha) = \max\{\Psi(DC(k,n-2-k, 2), \alpha): 0\le k\le n-2\},\]
where $\ell \ge \frac{n-2}{2}$. Then $\ell\in \{\lfloor\kappa_2\rfloor, \ \lceil \kappa_2\rceil\}$
where 
\[ \kappa_2 := \frac{n-2+\sqrt{(2t-1)^2(n-1)^2-2(n-1)+1}}{2}\approx tn.\]
Let $\varepsilon_2 = \kappa_2 - \lfloor \kappa_2\rfloor$. If $\ell = \lfloor\kappa_2\rfloor$, then 
\[ \Psi(DC(\ell, n-2-\ell, 2), \alpha) = \sqrt{(\alpha^2+(1-\alpha)^2)(n-1)}- \frac{(\alpha^2+(1-\alpha)^2)^{\frac{5}{2}}\ \varepsilon_2^2}{8\alpha^2(1-\alpha)^2(n-1)^{\frac{3}{2}}} + O_\alpha\left(\frac{1}{n^{\frac{5}{2}}}\right).\]
If $\ell = \lceil \kappa_2\rceil$ and $0<\varepsilon_2 < 1$, then 
\[ \Psi(DC(\ell, n-2-\ell, 2), \alpha) =\sqrt{(\alpha^2+(1-\alpha)^2)(n-1)}-\frac{(\alpha^2+(1-\alpha)^2)^{\frac{5}{2}}(1-\varepsilon_2)^2}{8\alpha^2(1-\alpha)^2(n-1)^{\frac{3}{2}}} + O_\alpha\left(\frac{1}{n^{\frac{5}{2}}}\right).\]
\end{proposition}

\begin{proof}
Define 
\[f(k) := \Psi(DC(k,n-2-k,2), \alpha).\]
Consider the function
\[ g(x):= \alpha\sqrt{x} + (1-\alpha)\sqrt{n-1-x}.\]
The derivative of $g(x)$ is
\[ g'(x)=\frac{\alpha}{2\sqrt{x}}-\frac{1-\alpha}{2\sqrt{n-1-x}}.\]
It is clear that $g'(x)=0$ precisely when 
\[ x = t(n-1).\]
Indeed, $g(x)$ is strictly increasing on $(0,t(n-1))$ and strictly decreasing on $(t(n-1), n-1)$.

Now, if one takes $z(k) = \frac{1}{2}\big(n-1+\sqrt{(n-1)^2 - 4k(n-2-k)} \big)$, then 
\[ \Psi(DC(k,n-2-k,2), \alpha) = g(z(k))\]
by Proposition \ref{prop:DC_path_2}. It is now clear that $\Psi(DC(k,n-2-k,2), \alpha)$ is maximum when 
\[ z(k) = t(n-1).\]
Solving for $k\ge \frac{n-2}{2}$ (treating $k$ as a real number), gives
\[ k = \kappa_2.\]
Since $k$ is an integer, we have $k\in \{\lfloor\kappa_2\rfloor, \ \lceil \kappa_2\rceil\}$. This proves the first half of the assertion.

Now, let $T = DC(\lfloor\kappa_2\rfloor, n-2-\lfloor\kappa_2\rfloor, 2)$. By Proposition \ref{prop:DC_path_2}, Taylor expansion and using value of $\kappa_2$, we have 
\begin{align*}
   & \qquad \lambda_1(T) \\
   & = \sqrt{ \frac{1}{2}\big(n-1+\sqrt{(n-1)^2 - 4(\kappa_2 -\varepsilon_2)(n-2-(\kappa_2-\varepsilon_2))} \big)} \\
   & = \sqrt{t(n-1)} - \frac{\varepsilon_2}{2\sqrt{t(n-1)}}+\frac{1}{(n-1)^{\frac{3}{2}}}\left(\frac{\varepsilon_2}{2(2t-1)^2\sqrt{t}}-\frac{\varepsilon_2^2}{8\ t^{\frac{3}{2}}}\right) + O_\alpha\left(\frac{1}{n^{\frac{5}{2}}}\right).
\end{align*}
Similarly, 
\begin{align*}
  & \qquad \lambda_2(T) \\
  & = \sqrt{ \frac{1}{2}\big(n-1 - \sqrt{(n-1)^2 - 4(\kappa_2 -\varepsilon_2)(n-2-(\kappa_2-\varepsilon_2))} \big)} \\
   & = \sqrt{(1-t)(n-1)} + \frac{\varepsilon_2}{2\sqrt{(1-t)(n-1)}}-\frac{1}{(n-1)^{\frac{3}{2}}}\left(\frac{\varepsilon_2}{2(2t-1)^2\sqrt{1-t}}+\frac{\varepsilon_2^2}{8(1-t)^{\frac{3}{2}}}\right) + O_\alpha\left(\frac{1}{n^{\frac{5}{2}}}\right).
\end{align*}
Thus, 
\begin{align*}
  \Psi(T, \alpha) & = \alpha \lambda_1(T) + (1-\alpha)\lambda_2(T) \\
  & = \sqrt{(\alpha^2+(1-\alpha)^2)(n-1)}-\frac{(\alpha^2 + (1-\alpha)^2)^{\frac{5}{2}}\ \varepsilon_2^2}{8\alpha^2(1-\alpha)^2(n-1)^{\frac{3}{2}}} + O_\alpha\left(\frac{1}{n^{\frac{5}{2}}}\right). 
\end{align*}
Similarly, one can determine $\Psi(T, \alpha)$ when $T = DC(\lceil\kappa_2\rceil, n-2-\lceil\kappa_2\rceil, 2)$. The proof is complete.
\end{proof}

\begin{proposition}\label{prop:DC_path_3_alpha} Fix $\alpha \in (\frac{1}{2}, 1)$. Suppose that 
\[ \Psi(DC(\ell,n-3-\ell, 3), \alpha) = \max\{\Psi(DC(k,n-3-k, 3), \alpha): 0\le k\le n-3\},\]
where $\ell\ge \frac{n-3}{2}$. Then $\ell\in \{\lfloor\kappa_3\rfloor, \ \lceil \kappa_3\rceil\}$
where 
\[ \kappa_3 = \frac{n-3+\sqrt{(2t-1)^2(n-1)^2-4}}{2}\approx tn.\]
Let $\varepsilon_3 = \kappa_3 - \lfloor \kappa_3\rfloor$. If $\ell = \lfloor\kappa_3\rfloor$, then
\[ \Psi(DC(\ell, n-3-\ell, 3), \alpha) = \sqrt{(\alpha^2+(1-\alpha)^2)(n-1)}-\frac{(\alpha^2 + (1-\alpha)^2)^{\frac{5}{2}}\ \varepsilon_3^2}{8\alpha^2(1-\alpha)^2(n-1)^{\frac{3}{2}}} + O_\alpha\left(\frac{1}{n^{\frac{5}{2}}}\right).\]
If $\ell = \lceil \kappa_3\rceil$ and $0<\varepsilon_3<1$, then 
\[ \Psi(DC(\ell, n-3-\ell, 3), \alpha) = 
\sqrt{(\alpha^2+(1-\alpha)^2)(n-1)}-\frac{(\alpha^2 + (1-\alpha)^2)^{\frac{5}{2}}\ (1-\varepsilon_3)^2}{8\alpha^2(1-\alpha)^2(n-1)^{\frac{3}{2}}} + O_\alpha\left(\frac{1}{n^{\frac{5}{2}}}\right).\]
\end{proposition}

\begin{proof}
    The proof is similar to the proof of Proposition \ref{prop:DC_path_2_alpha}.
\end{proof}

\subsection{Proof of Theorem \ref{thm:alpha_more_than_half}}

Fix $\alpha \in (\frac{1}{2}, 1)$ and let $t = \frac{\alpha^2}{\alpha^2 + (1-\alpha)^2}$. Let $T^*\in \mathcal{T}(n)$ be such that $\Psi(T^*,\alpha)=\Psi_n(\alpha)$.
Throughout this section, $\lambda_1^*$ and $\lambda_2^*$ will denote $\lambda_1(T^*)$ and $\lambda_2(T^*)$, respectively. Let $x$ denote the unit positive $\lambda_1^*$-eigenvector and $y$ denote a unit $\lambda_2^*$-eigenvector of $T^*$ with minimal support. Let $(H_1,a)$ and $(H_2,b)$ be rooted subtrees of $T^*$ w.r.t. $y$ as given in Theorem \ref{thm:spectral_center}.

We will prove Theorem \ref{thm:alpha_more_than_half} in a series of claims. Here's an outline of the proof.

\textbf{Proof outline.} We first show that  $\lambda_1^* \approx \sqrt{tn}$ and $\lambda_2^*\approx \sqrt{(1-t)n}$. 
We use this to find a vertex $u_1$ with $\deg(u_1)\approx tn - O(\sqrt{n})$ and a vertex $u_2$ with $\deg(u_2)\approx (1-t)n - O(\sqrt
n)$. This would imply that all other vertices have degree $O(\sqrt{n})$ and small eigen-entries in the eigenvectors $x$ and $y$. With further bootstrapping, we argue that $\deg(u_1)\approx tn - O(1)$ and $\deg(u_2)\approx (1-t)n -O(1)$. This enables us to show that $u_1$ and $u_2$ have large eigen-entries in $x$ and $y$. Finally, using several graph transformations, we will show that $T^*$ has to be a double comet with distance between $u_1$ and $u_2$ being at most $2$, completing the proof of Theorem \ref{thm:alpha_more_than_half}. 

We begin the proof by estimating the eigenvalues $\lambda_1^*$ and $\lambda_2^*$.

\begin{claim}\label{claim:lambda_1_2_estimates}
We have 
    \begin{enumerate}[$(i)$]
        \item $\lambda_1^* = \sqrt{t(n-1)} \pm O_\alpha\left(\frac{1}{\sqrt{n}}\right)$,
        \item $\lambda_2^* = \sqrt{(1-t)(n-1)} \pm O_\alpha\left(\frac{1}{\sqrt{n}}\right)$.
    \end{enumerate}
\end{claim}

\begin{proof}
By Proposition \ref{prop:DC_path_2_alpha}, we know that 
\[ \alpha \lambda_1^* + (1-\alpha)\lambda_2^* \ge \sqrt{(\alpha^2 + (1-\alpha)^2)(n-1)} - O\left(\frac{1}{n^{\frac{3}{2}}}\right).\]
Define 
\[f(x) := \alpha x + (1-\alpha)\sqrt{n-1 - x^2}\] 
(imagine $x = \lambda_1^*$). Clearly, $f(\lambda_1^*) \ge \alpha \lambda_1^* + (1-\alpha)\lambda_2^*.$ Hence, in order to find bounds for $\lambda_1^*$, we need to solve the inequality 
\[ f(x) \ge \sqrt{(\alpha^2 + (1-\alpha)^2)(n-1)} - O\left(\frac{1}{n^{\frac{3}{2}}}\right).\]
It is an elementary exercise to show that the above inequality holds whenever 
\[ x = \sqrt{t(n-1)} \pm O\left(\frac{1}{\sqrt{n}}\right).\]
This proves $(i)$. Assertion $(ii)$ can be proved similarly by solving the following inequality (imagine $x = \lambda_2^*$ this time):
\[ \alpha \sqrt{n-1-x^2} + (1-\alpha)x \ge \sqrt{\alpha^2 + (1-\alpha)^2(n-1)} - O\left(\frac{1}{n^{\frac{3}{2}}}\right). \qedhere\]\end{proof}

We now find two vertices with large degrees in $T^*$.

\begin{claim}\label{claim:u_1_u_2_degree_sqrtn}
    There exists $u_1\in V(H_1)$ and $u_2\in V(H_2)$ such that 
    \[\deg(u_1) = tn - O_\alpha(\sqrt{n}) \quad\text{and}\quad \deg(u_2) = (1-t)n - O_\alpha(\sqrt{n}).\]
    Moreover, if $u\notin \{u_1, u_2\}$, then $\deg(u) = O_\alpha(\sqrt{n})$.
\end{claim}

\begin{proof}
 Consider the $\lambda_1^*$-eigenvector $x$ of $T^*$. Let $u_1\in V(T^*)$ be such that $x_{\max} = x_{u_1}$. Using Lemma \ref{lemma:eigenvalue_equations}, we get 
 \begin{align*}
    \deg(u_1)x_{\max} & = (\lambda_1^*)^2x_{\max} - \sum_{d(u_1,w) = 2} x_w \\
    & \ge (\lambda_1^*)^2x_{\max} - \sum_{w\neq u_1} x_{\max}\frac{\deg(w)}{\lambda_1^*}\\
    & \ge (\lambda_1^*)^2x_{\max} - x_{\max} \frac{2n}{\lambda_1^*},
 \end{align*}
 which implies 
 \[\deg(u_1) \ge (\lambda_1^*)^2 - \frac{2n}{\lambda_1^*} \ge tn - O(\sqrt{n})\]
 using Claim \ref{claim:lambda_1_2_estimates}. Without loss of generality, we can assume that $u_1\in V(H_1)\cup S^0(y)$.
 
 Now, let $z$ be a unit positive $\lambda_1(H_2)$-eigenvector of $H_2$. Choose $u_2\in V(H_2)$ such that $z_{\max} = z_{u_2}$. Using Lemma \ref{lemma:eigenvalue_equations} and following the same steps as above, we get 
\[ \deg(u_2)\ge \deg_{H_2}(u_2)\ge \lambda_1^2(H_2) - \frac{2|H_2|}{\lambda_1(H_2)}.\]
Since $\lambda_1(H_2)\ge \lambda_2^*$ using Theorem \ref{thm:spectral_center}, we get 
\[ \deg(u_2) \ge (\lambda_2^*)^2 - \frac{2n}{\lambda_2^*} \ge (1-t)n - O(\sqrt{n}),\]
where the last inequality holds by Claim \ref{claim:lambda_1_2_estimates}. 

Now, if $u_1\in S^0(y)$, then $|H_1|= O(\sqrt{n})$, which implies 
\[ \sqrt{(1-t)n} - O\left(\frac{1}{\sqrt{n}}\right)\le \lambda_2^* \le \lambda_1(H_1)\le O(n^{\frac{1}{4}}),\]
a contradiction. Thus $u_1\in V(H_1)$. 

Finally, since $T^*$ is a tree, it is clear that the remaining vertices have degree at most $O(\sqrt{n})$. This completes the proof of the claim.
\end{proof}

Next, we show that the eigenvector $y$ attains maximum and minimum at $u_1$ and $u_2$, respectively.

\begin{claim}\label{claim:y_max_min}
    We have $y_{\max} = y_{u_1}$, $y_{\min} = y_{u_2}$ and $y_{u_1} > y_u > y_{u_2}$ for any vertex $u\notin \{u_1, u_2\}$.
\end{claim}

\begin{proof}
    Let $u\in V(H_1)$ be such that $y_{\max} = y_{u}$. Using Lemma \ref{lemma:eigenvalue_equations}, we have 
    \begin{align*}
        \deg(u)y_{\max} & = (\lambda_2^*)^2 y_{\max} - \sum_{\substack{d(w,u)=2\\ w\in S^+(y) }}y_w - \sum_{\substack{d(w,u)=2\\ w\in S^-(y) }}y_w\\
        & \ge (\lambda_2^*)^2 y_{\max} - \sum_{\substack{d(w,u)=2\\ w\in S^+(y) }}y_w\\
        & \ge (\lambda_2^*)^2 y_{\max} - \sum_{d(w,u)=2} y_{\max}\frac{\deg(w)}{\lambda_2^*}\\
        & \ge (\lambda_2^*)^2 y_{\max} - y_{\max}\frac{2n}{\lambda_2^*}, 
    \end{align*}
    which implies $\deg(u)\ge (1-t)n - O(\sqrt{n})$ using Claim \ref{claim:lambda_1_2_estimates}. By Claim \ref{claim:u_1_u_2_degree_sqrtn}, the only vertex in $H_1$ with degree linear in $n$ is $u_1$. It follows that $u=u_1$.  

    Using a similar argument, we can show that $y_{\min} = y_{u_2}$. It is clear from the proof that any vertex $u$ with $y_u\in \{y_{\max},\ y_{\min}\}$ must have large degree and therefore must lie in $\{u_1, u_2\}$. In other words, if $u\notin \{u_1, u_2\}$, then $y_{\max} > y_u > y_{\min}$.  
\end{proof}

In the next few claims, we estimate the entries of $x$ and $y$ and further improve the estimates for the vertex degrees.

\begin{claim}\label{claim:small_entry_remaining_vertices}
    If $u\notin \{u_1, u_2\}$, then $\max \{x_{u}, |y_u|\} = O_\alpha\left(\frac{1}{\sqrt{n}}\right)$.
\end{claim}

\begin{proof}
Using Lemma \ref{lemma:eigenvalue_equations}, we have
\begin{align*}
   |y_u| & = \frac{|\sum_{d(w,u)=2}y_w|}{(\lambda_2^*)^2 - \deg(u)}  \le \frac{\sum_{d(w,u)=2}|y_w|}{(\lambda_2^*)^2 - \deg(u)} \\
   & \le \frac{\sum_{d(w,u)=2}\frac{\deg(w)}{\lambda_2^*}}{(\lambda_2^*)^2 - \deg(u)}\le \frac{2n}{\lambda_2^*\left((\lambda_2^*)^2 - \deg(u)\right)}\\
   & = O\left(\frac{1}{\sqrt{n}}\right) 
\end{align*}
where the last inequality follows using Claims \ref{claim:lambda_1_2_estimates} and \ref{claim:u_1_u_2_degree_sqrtn}. A similar argument works for the eigenvector $x$.
\end{proof}

We further optimize the estimates for the degrees of $u_1$ and $u_2$.

\begin{claim}\label{claim:u_1_u_2_degree_constant}
   We have $\deg(u_1) = tn \pm O_\alpha(1)$ and $\deg(u_2) = (1-t)n \pm O_\alpha(1)$. Moreover, $\deg(u)=O_\alpha(1)$ for all $u\notin \{u_1, u_2\}$.
\end{claim}

\begin{proof}
Consider the unit positive $\lambda_1(H_2)$-eigenvector $z$ as in the proof of Claim \ref{claim:u_1_u_2_degree_sqrtn}. By the choice of $u_2$, we have $z_{\max} = z_{u_2}$. Let $u \in V(H_2)\backslash \{u_2\}$. As in the proof of Claim \ref{claim:small_entry_remaining_vertices}, one can argue that $z_u = O\left(\frac{z_{\max}}{\sqrt{n}}\right)$. 

Now, using Lemma \ref{lemma:eigenvalue_equations}, we have 
\begin{align*}
    \lambda_1(H_2)\left((\lambda_1(H_2))^2 - \deg_{H_2}(u_2)\right)z_{\max} & = \sum_{u\in N_{H_2}(u_2)}z_u (\deg_{H_2}(u)-1) + \sum_{d(u_2,u) = 3}z_u \\
    & \le \sum_{u \neq u_2} z_u \deg_{H_2}(u)\\
    & \le O\left(\frac{z_{\max}}{\sqrt{n}}\right)\cdot 2n = O(\sqrt{n}\ z_{\max}).
\end{align*}
It follows that
\[ \deg_{H_2}(u_2)  \ge \lambda_1^2(H_2) - \frac{O(\sqrt{n}\ z_{\max})}{\lambda_1(H_2)\ z_{\max}}\ge (\lambda_2^*)^2 - \frac{O(\sqrt{n})}{\lambda_2^*}\ge (1-t)n - O(1).\]
This also shows that $\deg(u_1)\le tn + O(1)$. Arguing as above with eigenvector $x$ shows that $\deg(u_1) \ge tn - O(1)$, which in turn also shows that $\deg(u_2) \le (1-t)n + O(1)$. The claim follows.
\end{proof}

\begin{claim}\label{claim:x_max_entry} We have $\min\{x_{u_1}, |y_{u_2}|\} = \Omega_\alpha(1)$.    
\end{claim}

\begin{proof} Since $x$ is a unit vector, we have
\begin{align*} 
1 & \le    x_{u_1}^2 + \sum_{\substack{v\in N(u_1)\\ \deg(v) =1}} x_v^2 + x_{u_2}^2 + \sum_{\substack{w\in N(u_2)\\ \deg(w) =1}} x_w^2 + O(1)\cdot O\left(\frac{1}{n}\right) \\
& \le x_{u_1}^2\left(1 + \frac{tn}{(\lambda_1^*)^2}\right) + x_{u_2}^2\left(1 + \frac{(1-t)n}{(\lambda_1^*)^2}\right) + O\left(\frac{1}{n}\right)\\
& \le \left(2 + \frac{t + (1-t)}{t}\right)x_{\max}^2 +  O\left(\frac{1}{n}\right),
\end{align*}
implying $x_{u_1}=x_{\max} = \Omega(1)$. A similar calculation shows that 
\[ 1 \le \left(2 + \frac{t + (1-t)}{(1-t)} \right)|y|_{\max}^2 + O\left(\frac{1}{n}\right),\]
implying $\max\{y_{u_1}, |y_{u_2}|\} = \Omega(1)$.

As $x$ and $y$ are orthogonal vectors, we have 
\begin{equation}\label{eq:x_y_orthogonal}
 \left|x_{u_1}y_{u_1}\left(1 + \frac{tn}{\lambda_1^*\ \lambda_2^*}\right) - x_{u_2}|y_{u_2}|\left(1 + \frac{(1-t)n}{\lambda_1^*\ \lambda_2^*}\right)\right| = O\left(\frac{1}{n}\right).    
\end{equation}

If $|y_{u_2}|\ge y_{u_1}$, then we are done. Else, if $|y_{u_2}| < y_{u_1}$, then by \eqref{eq:x_y_orthogonal}, we have 
\[ x_{u_2}|y_{u_2}| = \Omega(1)\]
implying $|y_{u_2}| = \Omega(1)$. The claim holds.
\end{proof}

\begin{claim}\label{claim:yu1_xu2_small_entry}
    We have $\max\{y_{u_1}, x_{u_2}\} = O_\alpha\left(\frac{1}{\sqrt{n}}\right)$.
\end{claim}

\begin{proof} Using Lemma \ref{lemma:eigenvalue_equations} $(i)$, we have 
\[ x_{u_2} = \frac{\sum_{d(u,u_2)=2} x_u}{(\lambda_1^*)^2 - \deg(u_2)} = \frac{O(\sqrt{n})}{(2t-1)n + O(\sqrt{n})} = O\left(\frac{1}{\sqrt{n}}\right).\]
A similar argument works for $y_{u_1}$. 
\end{proof}

We are now ready to make further structural claims about $T^*$. 

\begin{claim}\label{claim:large_H_1} We have $\lambda_1(H_1)>\lambda_1(H_2)$. Also, $T^*$ has a spectral edge $ab$.
\end{claim}

\begin{proof}
Using Claim \ref{claim:u_1_u_2_degree_constant}, we have 
\[ \lambda_1(H_1)\ge \sqrt{tn - O(1)} >\sqrt{(1-t)n + O(1)}\ge \lambda_1(H_2).\]
Using Theorem \ref{thm:spectral_center}, we conclude that $T^*$ does not have a spectral vertex.
\end{proof}

\begin{claim}\label{claim:u_1_u_2_caterpillar}
$T^*$ is a caterpillar on the path joining $u_1$ and $u_2$. 
\end{claim}

\begin{proof}
Let $P: u_1=v_1, \ldots, v_{k+1}=u_2$ be the path joining $u_1$ and $u_2$ in $T^*$. We claim that $T^*$ is a caterpillar on path $P$. Suppose not. Then there exists a vertex $v\notin \{u_1, u_2\}$ such that $\deg(v)\ge 2$ and $v$ has at most one neighbour (say $u$) which is not a leaf. Let $w$ be a leaf with neighbour $v$. Define $T=T^*-vw + uw$. We have 
\[ \lambda_2^* y_v = y_u + (\deg(v)-1)\frac{y_v}{\lambda_2^*},\]
which implies 
\[|y_v| < \left(\lambda_2^* - \frac{O(1)}{\lambda_2^*}\right)|y_v| = |y_u|.\] 
Also, note that $y_u$ and $y_v$ have the same sign. A similar argument shows that $x_v < x_u$. It follows that 
\[\alpha x_w(x_u - x_v) + (1-\alpha)y_w(y_u - y_v)>0.\]
Using Lemma \ref{lemma:rotation_alpha}, we see that $\Psi(T, \alpha)>\Psi(T^*, \alpha)$, a contradiction. 
\end{proof}

\begin{claim}\label{claim:H_1_star}
    $H_1$ is a star. 
\end{claim}

\begin{proof}
Suppose $H_1$ is not a star. Then there exist vertices $v$ and $w$ in $H_1$ such that $v\sim u_1$ and $v\sim w$. Define $T = T^* - vw + u_1w$. Using Claims \ref{claim:x_max_entry} and \ref{claim:y_max_min}, we see that 
\[ \alpha x_w(x_{u_1}-x_v) + (1-\alpha)y_w(y_{u_1}-y_v) > 0.\]
Using Lemma \ref{lemma:rotation_alpha}, we get $\Psi(T, \alpha) > \Psi(T^*, \alpha)$, a contradiction.
\end{proof}

Next, one would naturally like to show that $H_2$ is also a star, but this will require more work. In the subsequent few claims, we will argue that $T^*$ is a double comet with $d(u_1, u_2)\le 2$. We will consider the cases $a=u_1$ and $a\neq u_1$ separately. 

\begin{claim}\label{claim:a_not_u1}
    If $a\neq u_1$, then $T^*$ is a double comet with $d(u_1, u_2)=2$. 
\end{claim}

\begin{proof}
    Suppose not. Define $T = (H_1, a)\circ (K_{1, {|H_2|-1}},u_2)$ where $(K_{1, {|H_2|-1}},u_2)$ denotes a star on $|H_2|$ centred and rooted at $u_2$. Note that $\deg(a)=2$ using Claim \ref{claim:H_1_star}. Since $T$ can also be obtained by (repeatedly) applying Kelmans Operation on $T^*$ (in particular, on $H_2$), we see that $\lambda_1(T)>\lambda_1(T^*)$.
    Note that $\lambda_1(H_1-a) \ge \sqrt{tn-O(1)} > \sqrt{(1-t)n+O(1)} \ge \lambda_1(H_2)$ by Claim \ref{claim:u_1_u_2_degree_constant}. By the Interlacing Theorem, we have
    \[\lambda_2(T)\ge \lambda_2(T-a) = \min\{\lambda_1(H_1-a), \lambda_1(K_{1, {|H_2|-1}})\} \ge  \lambda_1(H_2)> \lambda_2^*.\]  
    This contradicts the choice of $T^*$. 
\end{proof}

\begin{claim}\label{claim:a_u1} Assume $a=u_1$. Then $T^*$ is a double comet with $d(u_1, u_2)\le 3$.  
\end{claim}

\begin{proof}
 Let $P: u_1=v_1, \ldots, v_{k+1}=u_2$ be the path joining $u_1$ and $u_2$ where $k=d(u_1, u_2)$. By Claim \ref{claim:u_1_u_2_caterpillar}, $T^*$ is a caterpillar on $P$. By assumption $v_1=a$ and $v_2=b$. If $k=1$, then clearly the claim holds. So assume $k\ge 2$. 
   
  We will first show that $\deg(v_2) = 2$. Suppose, to the contrary, that there is a leaf $w$ with neighbour $v_2$. Define $T=T^*-v_2w + u_1w$. Using Claim \ref{claim:small_entry_remaining_vertices}, observe that   
   \[ |y_{v_2}|= O\left(\frac{1}{\sqrt{n}}\right) \quad \text{and}\quad
   |y_w| = \frac{y_{v_2}}{\lambda_2^*} = O\left(\frac{1}{n}\right).\]
   Moreover, using Claim \ref{claim:x_max_entry}, we have 
   \[x_{v_2}\ge \frac{x_{u_1}}{\lambda_1^*}=\Omega\left(\frac{1}{\sqrt{n}}\right) \quad \text{and} \quad x_{w} = \frac{x_{v_2}}{\lambda_1^*}=\Omega\left(\frac{1}{n}\right).\]
   Using Claim \ref{claim:yu1_xu2_small_entry}, we have 
   \begin{align*}
     & \alpha x_w(x_{u_1}-x_{v_2}) + (1-\alpha)y_w(y_{u_1}-y_{v_2}) 
     \ge  \ \alpha \Omega\left(\frac{1}{n}\right) - (1-\alpha)O\left(\frac{1}{n^{\frac{3}{2}}}\right) > 0.
   \end{align*}
   By Lemma \ref{lemma:rotation_alpha}, we conclude that $\Psi(T, \alpha)>\Psi(T^*, \alpha)$, a contradiction. Therefore, $\deg(v_2)=2$. A similar argument also shows that $\deg(v_k)=2$. So, if $k\in \{2,3\}$, then claim holds. 
   
   Assume now that $k\ge 4$. We would like to arrive at a contradiction. Define $(H, b)$ to be the rooted tree (with root $b$) obtained from $(H_2, b)$ by applying Kelmans Operation from $v_{k-1}$ to $v_k$. Note that $\deg_H(v_k)\ge 3$ and $v_{k-1}$ is a leaf attached to $v_k$ in $H$. Furthermore, define $(H', b)$ to be the tree obtained from $(H, b)$ by contracting the edge $v_ku_2$ (i.e., move the neighbours of $v_k$ to $u_2$ and delete $v_k$). Finally, obtain $(H'', b)$ from $(H',b)$ by attaching a leaf (call it $w$) at vertex $v$, where 
   \[v=\begin{cases}
       v_3 & \text{if }k\ge 5;\\
       u_2 & \text{if }k=4.
   \end{cases}\]
   Define $T = (H_1, u_1)\circ (H'', b)$. We will argue that $\Psi(T, \alpha) > \Psi(T^*, \alpha)$. 

   Note that $T$ is obtained from $T^*$ by first applying Kelmans Operation, then contracting and later adding an edge. Using Lemmas \ref{lemma:Kelmans}, \ref{lemma:Contraction} and \ref{lemma:edge_addition}, we see that $\lambda_1(T)>\lambda_1^*$. Similarly, we have 
   \[ \lambda_1(H')>\lambda_1(H)>\lambda_1(H_2).\]
   Observe that $H''-b\cong H'$, and so $\lambda_1(H''-b)>\lambda_1(H_2)$. Using Interlacing Theorem and Claim \ref{claim:large_H_1}, we see that 
   \[\lambda_2(T) \ge \lambda_2(T-b)\ge \min\{\lambda_1(H_1), \lambda_1(H''-b)\} > \lambda_1(H_2).\]
   We conclude that $\Psi(T, \alpha)> \Psi(T^*, \alpha)$, giving us the desired contradiction.
\end{proof}

\begin{claim}\label{claim:double_comet_distance_two}
 $T^*$ is a double comet with $d(u_1, u_2)\le 2$. 
\end{claim}

\begin{proof}
    Suppose not. Then by Claims \ref{claim:a_not_u1} and \ref{claim:a_u1}, we have $d(u_1, u_2)= 3$. In particular, $T^* = DC(k_1, k_2, 4)$ for some $k_1 = tn \pm O(1) $ and $k_2 = (1-t)n\pm O(1)$. By Lemma \ref{lemma:char_edge}, the characteristic polynomial of $T^*$ is given by 
     \[\Phi(T^*, x) = x^{n-4}\left(x^2 - k_1-1\right)\left(x^2 - k_2-1\right) - x^{n-6}\left(x^2-k_1\right)\left(x^2-k_2\right).\]
     We see that $\Phi(T^*, x)>0$, whenever $x>\sqrt{k_1+1.1}$, and so $\lambda_1^*<\sqrt{k_1+1.1}$. Similarly, $\Phi(T^*, x)<0$, whenever $x\in \left(\sqrt{k_2+1.1}, \sqrt{k_1}\right)$. It follows that $\lambda_2^* < \sqrt{k_2+1.1}$. We conclude that 
     \begin{align*}
        \Psi(T^*, \alpha)  & \le \alpha\sqrt{k_1+1.1} + (1-\alpha)\sqrt{k_2+1.1}\\
        & = \alpha\sqrt{k_1+1.1} + (1-\alpha)\sqrt{n-4-k_1+1.1}\\
        & \le \sqrt{(\alpha^2+(1-\alpha)^2)(n-1.8)}\quad (\text{Cauchy-Schwarz})\\
        & = \sqrt{(\alpha^2+(1-\alpha)^2)(n-1)} - \frac{2\sqrt{(\alpha^2+(1-\alpha)^2)}}{5\sqrt{n-1}} + O\left(\frac{1}{n^{\frac{3}{2}}}\right).
     \end{align*}
    This contradicts the fact that 
    \[  \Psi(T^*, \alpha)\ge \sqrt{(\alpha^2+(1-\alpha)^2)(n-1)} - O\left(\frac{1}{n^{\frac{3}{2}}}\right)\]
    by Proposition \ref{prop:DC_path_2_alpha}. 
\end{proof}

\begin{claim}
    We have 
    \[ T^*\in \{DC(k, n-2-k, 2): k = \lfloor\kappa_2\rfloor, \ \lceil \kappa_2\rceil \}\cup \{DC(k, n-3-k, 3): k = \lfloor\kappa_3\rfloor, \ \lceil \kappa_3\rceil \}\]
    where $\kappa_2$ and $\kappa_3$ are as defined in Proposition \ref{prop:DC_path_2_alpha} and \ref{prop:DC_path_3_alpha}, respectively. 
\end{claim}

\begin{proof}
By Claim \ref{claim:double_comet_distance_two}, we see that 
\[T^* \in \{ DC(k,n-2-k,2), \ DC(\ell, n-3-\ell, 3) \}\]
for some $ k = tn + O_\alpha(1)$ and $\ell = tn + O_\alpha(1)$. So it only remains to optimize over these special comets which is done in Propositions \ref{prop:DC_path_2_alpha} and \ref{prop:DC_path_3_alpha}.
\end{proof}

The proof of Theorem \ref{thm:alpha_more_than_half} is complete.

\section{Asymptotics of $\Psi_n(\alpha)$}
\label{section:asymptotics_convex_combination}

In the previous sections, we determined the extremal trees for $\Psi_n(\alpha)$. Here, we analyze its limiting behaviour, which would give upper/lower bounds for $\Psi_n(\alpha)$. Recall that $\hat{\Psi}_n(\alpha) = \frac{\Psi_n(\alpha)}{\sqrt{n-1}}$ is a function in $\alpha$ from $[0, 1]\rightarrow [0,1]$. We observe the following.

\begin{proposition}\label{prop:Psi_piecewise_linear} For any $n\ge 2$, the map $\hat{\Psi}_n(\alpha)$ is a piecewise-linear increasing function.
\end{proposition}

\begin{proof}
First observe that for any fixed $T\in \mathcal{T}(n)$, the function \[\frac{\Psi(T, \alpha)}{\sqrt{n-1}} = \frac{\alpha \lambda_1(T) + (1-\alpha)\lambda_2(T)}{\sqrt{n-1}}\] is a linear increasing function in $\alpha$. Since $\mathcal{T}(n)$ only has finitely many trees, there exist points $0 =\alpha_0 < \alpha_1 <  \ldots < \alpha_k=1$ and trees $T_1, \ldots, T_k \in \mathcal{T}(n)$ such that $\hat{\Psi}_n(\alpha) = \Psi(T_i, \alpha)$
whenever $\alpha_{i-1}\le \alpha \le \alpha_i$ for all $1\le i\le k$. This proves the assertion. 
\end{proof}

The proof of Proposition \ref{prop:Psi_piecewise_linear} is illustrated in Figure \ref{fig:Psi_example_n_6} for $n = 6$.

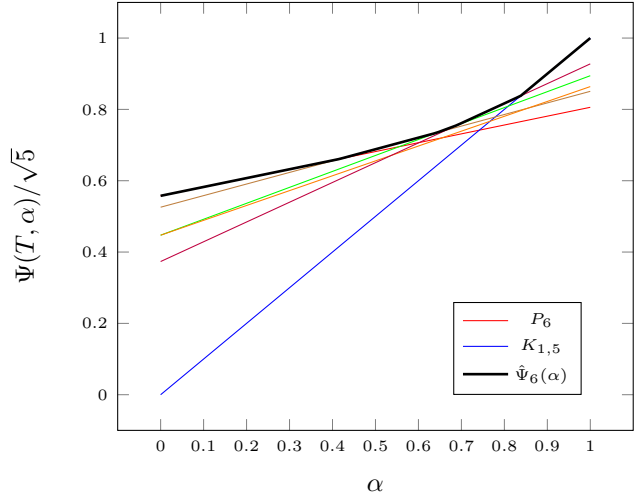
\begin{figure}
    \centering
    \begin{tikzpicture}[scale = 1]
        \begin{axis}[
        xlabel = {$\alpha$},
        ylabel = {\(\Psi(T, \alpha)/\sqrt{5}\)},
        xlabel style={font=\small},
        ylabel style={font=\small},
        xticklabel style={font=\tiny},
        yticklabel style={font=\tiny},
        xtick = {0, 0.1, 0.2, 0.3, 0.4, 0.5, 0.6, 0.7, 0.8, 0.9, 1},
        ytick = {0, 0.2, 0.4, 0.6, 0.8, 1},         legend style={at={(0.9,0.3)}, font =\tiny},
        ]
        \addplot [
        domain= 0:1, 
        samples=100, 
        color=red,
        ]{(1.80193*x + (1-x)*1.24697)/sqrt(5)};
        \addplot [
        domain= 0:1, 
        samples=100, 
        color=blue,
        ]{(2.23606*x + (1-x)*0)/sqrt(5)};
        \addplot [
        domain= 0:1, 
        samples=100, 
        color=brown,
        ]{(1.90211*x + (1-x)*1.17557)/sqrt(5)};
        \addplot [
        domain= 0:1, 
        samples=100, 
        color=green,
        ]{(2*x + (1-x)*1)/sqrt(5)};
        \addplot [
        domain= 0:1, 
        samples=100, 
        color=orange,
        ]{(1.93185*x + (1-x)*1)/sqrt(5)};
        \addplot [
        domain= 0:1, 
        samples=100, 
        color=purple,
        ]{(2.07431*x + (1-x)*0.83499)/sqrt(5)}; 
        \addplot [
        domain= 0:1, 
        samples=100, 
        color=black,
        line width = 1pt,
        ]{max((1.80193*x + (1-x)*1.24697), (2.23606*x + (1-x)*0), (1.90211*x + (1-x)*1.17557), (2*x + (1-x)*1), (1.93185*x + (1-x)*1), (2.07431*x + (1-x)*0.83499))/sqrt(5)};
        \legend{$P_6$, $K_{1,5}$, , , , , $\hat{\Psi}_6(\alpha)$}
        \end{axis}
    \end{tikzpicture}
    \caption{The black curve represents $\hat{\Psi}_6(\alpha)$. The different colored lines represent $\frac{\Psi(T, \alpha)}{\sqrt{5}}$ for different trees $T$ on 6 vertices.}
    \label{fig:Psi_example_n_6}
\end{figure}

We now find the (point-wise) limit of $\hat{\Psi}_n(\alpha)$ as $n\rightarrow \infty$, thus proving Theorem \ref{thm:asymptotic_convex_combination} 

\begin{proof}[Proof of Theorem \ref{thm:asymptotic_convex_combination}]
When $\alpha = 0, \frac{1}{2}, 1$, the assertion is clear.

Now, suppose that $0 < \alpha < \frac{1}{2}$. From the proof of Theorem \ref{thm:alpha_less_than_half}, we know that the extremal tree $T^*$ for $\Psi_n(\alpha)$ has \[\lambda_i^* = \sqrt{\frac{n}{2}\pm O(1)} \qquad (i=1,2).\] 
It follows that 
\[\lim_{n \rightarrow \infty}\hat{\Psi}_n(\alpha) = \lim_{n\rightarrow \infty} \frac{\sqrt{\frac{n}{2}\pm O(1)}}{\sqrt{n-1}} = \sqrt{\frac{1}{2}}.\]

Next, suppose that $\frac{1}{2}<\alpha< 1$. Using Theorem \ref{thm:alpha_more_than_half}, and Propositions \ref{prop:DC_path_2_alpha} and \ref{prop:DC_path_3_alpha}, we have \[\Psi_n(\alpha) = \sqrt{(\alpha^2+(1-\alpha)^2)(n-1)} + o(\sqrt{n}).\] 
We conclude that 
\[ \lim_{n\rightarrow \infty} \hat{\Psi}_n(\alpha) = \lim_{n\rightarrow \infty}\frac{\Psi_n(\alpha)}{\sqrt{n-1}} = \sqrt{\alpha^2 + (1-\alpha)^2}.\qedhere\]
\end{proof}

We illustrate how quickly $\hat{\Psi}_n(\alpha)$ approaches its limit by comparing $\hat{\Psi}_{26}(\alpha)$ with $\lim_{n\to \infty}\hat{\Psi}_n(\alpha)$ in Figure \ref{fig:Psi_limit_vs_n_26}.

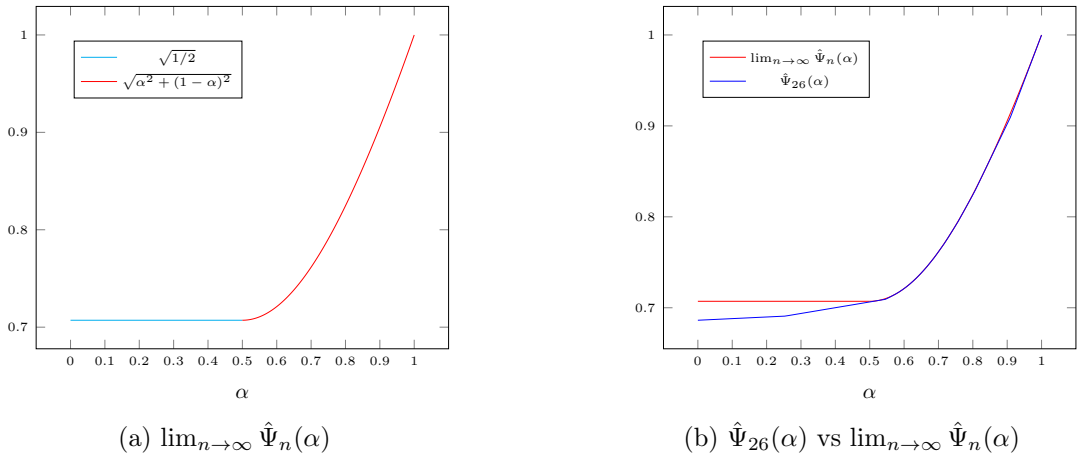
\begin{figure}
\begin{subfigure}{0.5\textwidth}
    \centering 
    \begin{tikzpicture}[scale = 0.8]
    \begin{axis}[
        xlabel = {$\alpha$},
        xlabel style={font=\small},
        ylabel style={font=\small},
        xticklabel style={font=\tiny},
        yticklabel style={font=\tiny},
        xtick = {0, 0.1, 0.2, 0.3, 0.4, 0.5, 0.6, 0.7, 0.8, 0.9, 1},
        ytick = {0, 0.1, 0.2, 0.3, 0.4, 0.5, 0.6, 0.7, 0.8, 0.9, 1},        
        legend style={at={(0.5,0.9)}, font =\tiny},
        ]
        \addplot[
        domain= 0:0.5, 
        samples=100, 
        color=cyan,
        ]{sqrt(1/2)};
        \addlegendentry{$\sqrt{1/2}$}
        \addplot[
        domain= 0.5:1, 
        samples=100, 
        color=red,
        ]{sqrt(x^2 + (1-x)^2)};
        \addlegendentry{$\sqrt{\alpha^2 + (1-\alpha)^2}$} 
        \end{axis}
    \end{tikzpicture}
    \subcaption{$\lim_{n\to \infty}\hat{\Psi}_n(\alpha)$}
\end{subfigure}
\begin{subfigure}{0.5\textwidth}
    \centering
    \begin{tikzpicture}[scale = 0.8]
    \begin{axis}[
        xlabel = {$\alpha$},
        xlabel style={font=\small},
        ylabel style={font=\small},
        xticklabel style={font=\tiny},
        yticklabel style={font=\tiny},
        xtick = {0, 0.1, 0.2, 0.3, 0.4, 0.5, 0.6, 0.7, 0.8, 0.9, 1},
        ytick = {0, 0.1, 0.2, 0.3, 0.4, 0.5, 0.6, 0.7, 0.8, 0.9, 1},        
        legend style={at={(0.5,0.9)}, font =\tiny},
        ]
        \addplot[
        domain= 0:0.5, 
        samples=100, 
        color=red,
        ]{sqrt(1/2)};
        \addplot[
        domain= 0.5:1, 
        samples=100, 
        color=red,
        ]{sqrt(x^2 + (1-x)^2)};
        \addlegendentry{limit}
        \addplot [
        domain= 0:1, 
        samples=100, 
        color=blue,
        ]{max(
        (x*5 + (1-x)*0),
        (x*4.903406609757669 + (1-x)*0.9780611531927870),
        (x*4.805705988739693 + (1-x)*1.380286184018175),
        (x*4.707080194548844 + (1-x)*1.686237243713357),
        (x*4.607832961196238 + (1-x)*1.941101594897472),
        (x*4.508462922244040 + (1-x)*2.161888544479279),
        (x*4.409787069091307 + (1-x)*2.356645498431000),
        (x*4.313151725579202 + (1-x)*2.529174211503263),
        (x*4.220790554667138 + (1-x)*2.680471431228596),
        (x*4.136396043495647 + (1-x)*2.808954925119582),
        (x*4.065849096332680 + (1-x)*2.910132492834429 ),
        (x*4.017468723542258 + (1-x)*2.976565983706685 ),
        (x*4 + (1-x)*3),
        (x*3.690262048791372 + (1-x)* 3.373716942965149),
        (x*3.520892626084280+ (1-x)*3.431375296157698)
        )/sqrt(25)};
        \legend{$\lim_{n\to \infty}\hat{\Psi}_n(\alpha)$, ,$\hat{\Psi}_{26}(\alpha)$}
        \end{axis}
    \end{tikzpicture}
    \subcaption{$\hat{\Psi}_{26}(\alpha)$ vs $\lim_{n\to \infty}\hat{\Psi}_n(\alpha)$}
\end{subfigure}
    \caption{Limit of $\hat{\Psi}_n(\alpha)$}
    \label{fig:Psi_limit_vs_n_26}
\end{figure}

\section{Minimum spectral sum of trees}
\label{section:minimum_spectral_sum}

In this section, we prove Theorem \ref{thm:spectral_sum_min_trees}. For $n\le 15$, it is verified computationally that the star minimizes the spectral sum in $\mathcal{T}(n)$. For $16\le n\le 18$, we verify that the path minimizes the spectral sum. The SageMath code for verification is \href{https://github.com/Shivaramkratos/Codes_misc/blob/main/convex_sum_specific_check.sage}{here}.

So assume $n\ge 19$. Let $T^{\#}\in \mathcal{T}(n)$ be such that \[ \lambda_1(T^{\#}) + \lambda_2(T^{\#}) = \min\{\lambda_1(T) + \lambda_2(T) : T\in \mathcal{T}(n)\}.\]
Throughout this section, $\lambda_1^{\#}$ and $\lambda_2^{\#}$ will denote $\lambda_1(T^{\#})$ and $\lambda_2(T^{\#})$, respectively. Clearly, 
\begin{equation}\label{eq:path_min_spectralsum}
\lambda_1^{\#} + \lambda_2^{\#} \le \lambda_1(P_n) + \lambda_2(P_n) = 2\cos\left(\frac{\pi}{n+1}\right) + 2\cos\left(\frac{2\pi}{n+1}\right)< 4,
\end{equation}
using Proposition \ref{prop:path_spectrum}. We now prove a series of claims, culminating in a proof of Theorem \ref{thm:spectral_sum_min_trees}.

\begin{claim}
    $\Delta(T^{\#})\le 3$.
\end{claim} 

\begin{proof}
    For any tree $T$ of order $18$ with $\Delta(T)\ge 4$, it can be checked (see the SageMath code \href{https://github.com/Shivaramkratos/Codes_misc/blob/main/convex_sum_specific_check.sage}{here}) that $\lambda_1(T)+\lambda_2(T)\ge 4$. So if $\Delta(T^{\#})\ge 4$, then one can find an induced subtree of order $18$ with maximum degree at least $4$ in $T^{\#}$ implying $\lambda_1^{\#}+\lambda_2^{\#}\ge 4$ by Interlacing Theorem. This contradicts \eqref{eq:path_min_spectralsum}.
\end{proof}

Let $y$ denote a unit $\lambda_2^{\#}$-eigenvector of $T^{\#}$ with minimal support, and $(H_1,a)$ and $(H_2,b)$ be rooted subtrees of $T^{\#}$ as given in Theorem \ref{thm:spectral_center}. 

\begin{claim}\label{claim:H_i_is_a_path}
    $H_1$ and $H_2$ are paths.
\end{claim}

\begin{proof} We will prove the assertion for $H_1$; a similar argument works for $H_2$.

Suppose, to the contrary, that $H_1$ is not a path. Then there exists $u\in V(H_1)$ such that $\deg(u)=3$. Also, $u\neq a$ since $\Delta(T^{\#})\le 3$ and $a$ has a neighbour outside $H_1$. Moreover, we can choose $u$ such that the components of $H_1-u$ not containing $a$ are paths. Let $N(u)=\{u_1, u_2, u_3\}$ where $u_3$ lies on the path joining $u$ and $a$ ($u_3$ could be $a$). Let $w_i$ be the terminal vertices of the paths starting at $u_i$ and not containing $u$ for $i\in \{1,2\}$, see Figure \ref{fig:smaller_spectral_sum_1}. Define $T' = T^{\#} - uu_2 + w_1u_2$. We will argue that $\lambda_1(T') + \lambda_2(T') < \lambda_1^{\#} + \lambda_2^{\#}$, giving us the desired contradiction.

It is clear that $\lambda_1^{\#} > \lambda_1(T')$ using the Hanging Path Lemma \ref{lemma:hanging_path}. So, it suffices to show that $\lambda_2^{\#} \ge \lambda_2(T')$. We consider the following cases:

\textbf{Case 1:} $T^{\#}$ has a spectral vertex $v$.

Let $H_1'$ denote the component of $T'-v$ containing $u$. Then, by the Hanging Path Lemma \ref{lemma:hanging_path}, we have $\lambda_1(H_1')<\lambda_1(H_1)$. Using Interlacing Theorem, we see that
\[ \lambda_2(T')\le \max\{\lambda_1(H_1'), \lambda_1(H_2)\} \le \max\{\lambda_1(H_1), \lambda_1(H_2)\} = \lambda_2^{\#}.\]

\textbf{Case 2:} $T^{\#}$ has a spectral edge $ab$.

Let $B$ be the component of $T^{\#} - uu_3$ (and also $T'-uu_3$) containing $u_3$. Let the path $u_1 \ldots w_1 \cong P_k$ and $u_2\ldots w_2 \cong P_{\ell}$. Now, using Lemma \ref{lemma:char_edge} with edge $uu_3$, we get 
\begin{align*}
    \Phi(T^{\#}, x) & = \Phi(B,x)\Phi(P_{k + \ell + 1},x) - \Phi(B-u_3, x)\Phi(P_k,x)\Phi(P_{\ell},x),\\
    \Phi(T', x) & = \Phi(B,x)\Phi(P_{k + \ell + 1},x) - \Phi(B-u_3, x)\Phi(P_{k+\ell},x).
\end{align*}
Hence, 
\begin{align*}
    \Phi(T', \lambda_2^{\#}) & = \Phi(T', \lambda_2^{\#}) - \Phi(T^{\#}, \lambda_2^{\#}) \\
    & =\Phi(B-u_3, \lambda_2^{\#})\left[ \Phi(P_k, \lambda_2^{\#})\Phi(P_{\ell}, \lambda_2^{\#}) - \Phi(P_{k+\ell}, \lambda_2^{\#}) \right]\\
    & = \Phi(B-u_3, \lambda_2^{\#})\Phi(P_{k-1}, \lambda_2^{\#})\Phi(P_{\ell-1}, \lambda_2^{\#}).
\end{align*}

It is easily seen that $\lambda_1(B-u_3) > \lambda_2^{\#} > \lambda_2(B-u_3)$, which implies $\Phi(B-u_3, \lambda_2^{\#})<0$. Also, $\Phi(P_{k-1}, \lambda_2^{\#})>0$ and $\Phi(P_{\ell-1}, \lambda_2^{\#}) > 0$, since $\lambda_2^{\#}> \lambda_1(H_1 - a)\ge \lambda_1(P_{k+\ell+1})$. It follows that $\Phi(T', \lambda_2^{\#})< 0$. Since $\lambda_3(T') < \lambda_2^{\#} < \lambda_1(T')$, we conclude that $\lambda_2(T')<\lambda_2^{\#}$. This completes the proof.
\end{proof}

\begin{figure}
\begin{subfigure}{0.49\textwidth}
\centering
\begin{tikzpicture}[scale = 0.7]

\draw [rotate around={0:(3.21,0)}, fill=gray!10] (3.21,0) ellipse (2.2754913851752945cm and 1.404906062342599cm);

\draw (0,0)-- (-0.78,0.78);
\draw (0,0)-- (-0.76,-0.78);
\draw (0,0)-- (1,0);
\draw [dashed] (1,0)-- (3,0);
\draw [dashed] (-0.78,0.78)-- (-2,2);
\draw [dashed] (-0.76,-0.78)-- (-2,-2);

\draw [fill=black] (0,0) circle (2pt);
\draw (0,0.3) node {$u$};
\draw [fill=black] (-0.78,0.78) circle (2pt);
\draw (-0.65,1) node {$u_1$};
\draw [fill=black] (-0.76,-0.78) circle (2pt);
\draw (-0.65,-1.1) node {$u_2$};
\draw [fill=black] (1,0) circle (2pt);
\draw (1,0.3) node {$u_3$};
\draw [fill=black] (3,0) circle (2pt);
\draw (3,0.3) node {$a$};
\draw [fill=black] (-2,2) circle (2pt);
\draw (-2, 2.3) node {$w_1$};
\draw [fill=black] (-2,-2) circle (2pt);
\draw (-2,-2.3) node {$w_2$};
\draw (3,1.1) node {$B$};
\end{tikzpicture}
\caption{$T^{\#}$}
\end{subfigure}
\begin{subfigure}{0.49\textwidth}
\centering
\begin{tikzpicture}[scale = 0.7]

\draw [rotate around={0:(3.21,0)}, fill=gray!10] (3.21,0) ellipse (2.2754913851752945cm and 1.404906062342599cm);

\draw (0,0)-- (-1,0);
\draw[dashed] (0,0)-- (-2, 0);
\draw (0,0)-- (1,0);
\draw [dashed] (1,0)-- (3,0);
\draw (-2,0)-- (-3,0);
\draw [dashed] (-3,0)-- (-4,0);

\draw [fill=black] (0,0) circle (2pt);
\draw (0,0.3) node {$u$};
\draw [fill=black] (-1,0) circle (2pt);
\draw (-1,0.3) node {$u_1$};
\draw [fill=black] (-3,0) circle (2pt);
\draw (-3,0.3) node {$u_2$};
\draw [fill=black] (1,0) circle (2pt);
\draw (1,0.3) node {$u_3$};
\draw [fill=black] (3,0) circle (2pt);
\draw (3,0.3) node {$a$};
\draw [fill=black] (-2,0) circle (2pt);
\draw (-2, 0.3) node {$w_1$};
\draw [fill=black] (-4,0) circle (2pt);
\draw (-4,0.3) node {$w_2$};
\draw (3,1.1) node {$B$};
\end{tikzpicture} 
\caption{$T'$}
\end{subfigure}
  \caption{$T^{\#}$ and $T'$}
  \label{fig:smaller_spectral_sum_1}
\end{figure}

We are now ready to finish the proof. We treat the spectral vertex and spectral edge cases separately.

\begin{claim}
    If $T^{\#}$ has a spectral vertex $v$, then $T^{\#}$ is a path.
\end{claim}

\begin{proof}

In light of Claim \ref{claim:H_i_is_a_path}, every vertex in $H_1$ (resp. $H_2$), with the possible exception of vertex $a$ (resp. $b$), has degree at most 2. Suppose $\deg(a)=3$. Let $N(a) = \{u_1, u_2, v\}$ and $w_i$ be the terminal vertices of the path starting at $a$ and containing $u_i$ for $i=1,2$. Define $T' = T^{\#} - au_2 + u_2w_1$. By Hanging Path Lemma \ref{lemma:hanging_path}, we have $\lambda_1(T')<\lambda_1^{\#}$. Also, by the Interlacing Theorem, we have 
\[ \lambda_2(T') \le \lambda_1(T'-v) = \lambda_2^{\#}.\]
We see that $T'$ has a smaller spectral sum than $T^{\#}$, a contradiction. We conclude that $\deg(a)=2$, and similarly $\deg(b)=2$.

To show that $T^{\#}$ is a path, we only need to argue that $\deg(v)=2$. Suppose, to the contrary, that $\deg(v)=3$. Note that  $\lambda_1(H_1)=\lambda_1(H_2)$ and both $H_1$ and $H_2$ are paths, which implies that $|H_1|=|H_2|$. Let $H$ denote the subtree $T^{\#}[S^0(y)]$ rooted at $v$. If $H$ is not a path, then define $T$ to be the tree obtained from $T^{\#}$ by deleting the vertices in $S^0(y)\backslash\{v\}$ and attaching a path at $v$ so that $|T| = |T^{\#}|$. Using the Hanging Path Lemma, we see that $\lambda_1(T)<\lambda_1(T^{\#})$. Clearly, $\lambda_2(T) = \lambda_1(T-v) = \lambda_2^{\#}$. Thus, $T$ has smaller spectral sum, a contradiction. Thus, $H$ is a path, i.e., $T^{\#}$ is a subdivision of $K_{1,3}$ with central vertex $v$. We have $\lambda_1(H_i)\ge \lambda_1(H-v)$, which implies $|H_i|\ge |H-v|$ for $i=1,2$. 

If $n\ge 60$, then $|H_i|\ge 20$. Let $T$ denote the tree obtained from $P_{27}=v_1\ldots v_{27}$ by attaching a leaf at the vertex $v_{14}$. Clearly, $T$ is an induced subgraph of $T^{\#}$, and so $\lambda_1^{\#} + \lambda_2^{\#}\ge \lambda_1(T) + \lambda_2(T)>4$,
a contradiction. 

When $19\le n < 60$, we verify the assertion computationally by considering all subdivisions of $K_{1,3}$ on $n$ vertices (see \href{https://github.com/Shivaramkratos/Codes_misc/blob/main/Spec_sum_minimize_check.sage}{here} for the code).
\end{proof}

\begin{claim}
If $T^{\#}$ has a spectral edge $ab$, then $T^{\#}$ is a path.
\end{claim}

\begin{proof}
In light of Claim \ref{claim:H_i_is_a_path}, the vertices of degree 3 in $T^{\#}$ belong to the set $\{a,b\}$.
Now, denote by $k$ (resp. $\ell$) the order of the longest path starting from $a$ (resp. $b$) and not containing $b$ (resp. $a$). Without loss of generality, let $\ell \ge k$. Since $H_1$ and $H_2$ are paths by Claim \ref{claim:H_i_is_a_path}, we have 
\[\lambda_1(P_{2k-1})\ge \lambda_1(H_1) > \lambda_1(H_2 - b)\ge \lambda_1(P_{\ell-1}),\]
which implies $k\ge \frac{\ell}{2}$. 

Assume $n\ge 60$. Then $n \le 2k+ 2\ell \le 6k$, which implies $k\ge 10$. If $\deg(a)=\deg(b)=3$, then the tree $T$ obtained from the path $P_{15} = v_1\ldots v_{15}$ by attaching a leaf at vertex $v_8$ and $v_9$ is an induced subgraph of $T^{\#}$. Since $T$ has spectral sum more than 4 (checked using SageMath), we get a contradiction. If $\deg(a)=3$ and $\deg(b)=2$, then $60\le 2k+\ell$, which gives $k\ge 15$ and $\ell \ge 20$. Similarly, if $\deg(a)=2$ and $\deg(b)=3$, then $60\le k+2\ell$, which gives $k\ge 12$ and $\ell \ge 20$. In either case, the tree $T$ obtained from the path $P_{30} = v_1\ldots v_{30}$ by attaching a leaf at vertex $v_{11}$ is an induced subgraph of $T^{\#}$. Again, $T$ has spectral sum greater than 4 (checked using SageMath), which leads to a contradiction.

When $19\le n < 60$, then the claim is verified computationally. By the above discussion, if $T^{\#}$ is not a path, then it is a subdivision of $K_{1,3}$ or $DC(2,2,2)$ (where only the pendant edges are subdivided). See \href{https://github.com/Shivaramkratos/Codes_misc/blob/main/Spec_sum_minimize_check.sage}{here} for the code we use to rule out these possibilities. 
\end{proof}

This completes the proof of Theorem \ref{thm:spectral_sum_min_trees}. 

\section{Concluding remarks}
\label{section:conclusion}

In this paper, we investigated the spectral sum and the convex combination of $\lambda_1$ and $\lambda_2$ of trees. However, similar questions for graphs remain open; for instance, the following conjecture by Ebrahimi, Mohar, Nikiforov, and Ahmady \cite{Ebrahimi_Mohar_Nikiforov_Ahmady_2008}. 

\begin{conjecture}[\cite{Ebrahimi_Mohar_Nikiforov_Ahmady_2008}, cf. Conjecture 17 in \cite{Aouchiche_Hansen_2010}]\label{conj:spectral_sum} For any graph $G$ of order $n$, we have
    \[ \lambda_1(G)+\lambda_2(G)\le \frac{8n}{7}.\] 
\end{conjecture}

We also believe that the ideas used in this paper may be useful in proving Conjecture \ref{conj:spectral_gap_trees}. There is an analogous conjecture for the spectral gap of graphs by Stani\'{c} \cite{Stanic_2013} which is wide open.

Finally, it is a natural problem to investigate the quantity $\min_{T\in \mathcal{T}(n)} \Psi(T, \alpha)$ for $\alpha \in [0,1]$. The answer is known for $\alpha = 0$, $\alpha = 1$ (Theorem \ref{thm:lambda_one_max_trees}), and $\alpha = \frac{1}{2}$ (Theorem \ref{thm:spectral_sum_min_trees}). We pose the following problem.

\begin{problem} Suppose $n$ is sufficiently large. Let $T\in \mathcal{T}(n)$ be such that 
\[\Psi(T, \alpha) = \min_{T\in \mathcal{T}(n)} \Psi(T, \alpha).\] 
Is it true that 
\begin{enumerate}[$(i)$]
    \item $T$ is a path if $\alpha > \frac{1}{2}$, and
    \item $T$ has at most one vertex of degree more than 2 if $\alpha < \frac{1}{2}$?
\end{enumerate}
\end{problem}

\section*{Acknowledgements}

Supported in part by the NSERC Discovery Grant R832714 (Canada), and in part by the ERC Synergy grant KARST (European Union, ERC, KARST, project number 101071836).

\bibliographystyle{plain}
\bibliography{references.bib}

\vspace{0.4cm}

\affl{Hitesh Kumar}{hitesh.kumar.math@gmail.com, hitesh\_kumar@sfu.ca}{Department of Mathematics, Simon Fraser University, Burnaby, Canada}

\affl{Bojan Mohar}{mohar@sfu.ca}{Department of Mathematics, Simon Fraser University, Burnaby, Canada\\On leave from FMF, Department of Mathematics, University of Ljubljana.}
 
\affl{Shivaramakrishna Pragada}{shivaramakrishna\_pragada@sfu.ca}{Department of Mathematics, Simon Fraser University, Burnaby, Canada}

\affl{Hanmeng Zhan}{hzhan@wpi.edu}{Computer Science Department, Worcester Polytechnic Institute, Worcester, MA 01609, United States}

\end{document}